\RequirePackage{fix-cm}

\documentclass[smallextended,final]{svjour3}       
\smartqed  
\usepackage{graphicx}
\usepackage{mathptmx}      
\usepackage{amsmath,latexsym,amsxtra,enumerate,color,cancel}
\usepackage{bbm, dsfont}
\usepackage{enumitem}
\usepackage[utf8]{inputenc}
\usepackage[english]{babel}
\usepackage[usenames,dvipsnames,svgnames,table]{xcolor}
\usepackage{times, amsfonts, amssymb,  amscd, graphicx,stmaryrd}
\usepackage{mathrsfs}
\usepackage{algorithm, algorithmicx, algpseudocode}
\usepackage{graphicx}
\usepackage[active]{srcltx}
\usepackage{xspace}
\usepackage{a4wide}

\usepackage{aliascnt}
\RequirePackage{hyperref}

\makeatletter
\def\namedlabel#1#2{\begingroup
	#2%
	\def\@currentlabel{#2}%
	\phantomsection\label{#1}\endgroup
}
\makeatother

\newcommand{\N}{\mathbb{N}}
\newcommand{\E}{\mathbb{E}}
\newcommand{\R}{\mathbb{R}}
\newcommand{\Rd}{\mathbb{R}^d}
\newcommand{\Rzerod}{\mathbb{R}_0^d}

\newcommand{\Rdk}{\mathbb{R}^{d \times k}}
\newcommand{\PP}{\mathbb{P}}
\newcommand{\F}{\cF}
\newcommand{\cF}{\mathcal{F}}
\newcommand{\equa}{\begin{eqnarray*}}
	\newcommand{\tion}{\end{eqnarray*}}
\newcommand{\equal}{\begin{eqnarray}}
\newcommand{\tionl}{\end{eqnarray}}

\newcommand{\czero}{c_0}
\newcommand{\cone}{c_1}
\newcommand{\ctwo}{c_2}
\newcommand{\cthree}{c_3}
\newcommand{\cfour}{c_4}
\newcommand{\cfive}{c_5}

\newcommand{\ietat}{e^{\int_0^t\eta(\tau)d\tau}}
\newcommand{\ietas}{e^{\int_0^s\eta(\tau)d\tau}}
\newcommand{\ietaT}{e^{\int_0^T\eta(\tau)d\tau}}

\newcommand{\fsyzu}{f(s,Y_s,Z_s,U_s)}

\newcommand{\Kfzero}{K_{|\fzero|}}
\newcommand{\Ifzero}{I_{|\fzero|}}
\newcommand{\fzero}{f_0}
\newcommand{\fzerot}{\fzero(t)}

\newcommand{\inttTR}{\int_{{]t,T]}\times\R_0}}

\newcommand{\filtration}{(\mathcal{F}_t)_{t \in [0,T]}}
\newcommand{\sequence}[1]{(#1_n)_{n \in \N}}

\newcommand{\innerproduct}[2]{{\langle #1 , #2 \rangle}}

\newcommand{\tempinnerproduct}[2]{{\langle #1 , #2 \rangle}}

\usepackage{tcolorbox}

\newcommand{\changed}[1]{
	\begin{tcolorbox}[colframe=black,
		colback=SpringGreen,
		title=Changed]
		#1
	\end{tcolorbox}
}
\renewcommand{\changed}[1]{#1}

\DeclareMathOperator{\esssup}{esssup}
\DeclareMathOperator{\trace}{trace}

\newcommand{\levy}{L\'evy\xspace}
\newcommand{\ito}{It\^{o}\xspace}
\newcommand{\lp}{L^p}
\newcommand{\lloc}{L_{loc}}
\newcommand{\mloc}{\mathcal{M}_{loc}}

\begin{document}

\title{$L^p$-Solutions and                              
	Comparison Results for L\'evy Driven BSDEs in a Monotonic,        
	General Growth Setting}


\author{Stefan Kremsner *         \and
        Alexander Steinicke 
}


\institute{Stefan Kremsner \at
              Department of Mathematics, University of Graz, Austria. \\
              \email{stefan.kremsner@uni-graz.at}             \\
			   * Corresponding author
           \and
           Alexander Steinicke \at
              Department of Mathematics, Montanuniversitaet Leoben, Austria. \\
              \email{alexander.steinicke@unileoben.ac.at}           
}

\date{}

\def\makeheadbox{\relax}

\maketitle

\begin{abstract}
We present a unified approach to $L^p$-solutions ($p > 1$) of multidimensional backward stochastic differential equations (BSDEs) driven by L\'evy processes and more general filtrations. New existence, uniqueness and comparison results are obtained. The generator functions obey a time-dependent extended monotonicity (Osgood) condition in the $y$-variable and have general growth in $y$.  Within this setting, the results generalize those of Royer, Yin and Mao, Yao,
Kruse and Popier and Geiss and Steinicke.
\end{abstract}

\section{Introduction}
\label{intro}

The existence and uniqueness of solutions to a backward stochastic differential equation (BSDE) has been extensively investigated in many, but also various specifically chosen settings, partly due to certain applications in practice and partly also for theoretically interesting reasons. 
In this paper we both unify and simplify the approach for a general BSDE framework driven by a \levy process with a straightforward extension to more general filtrations. We show new comparison results and relax the assumptions known so far for guaranteeing unique $\lp$-solutions, $p > 1$, to a BSDE with terminal condition $\xi$ and generator $f$ that satisfies a monotonicity condition. An $\lp$-solution is a triplet \changed{of processes} $(Y,Z,U)$ from suitable $L^p$-spaces (defined in \autoref{sec:setting}) which satisfies a.s.
\begin{align}\label{eq_lp_solution_informal}
Y_t=\xi+\int_t^T f(s,Y_s,Z_s,U_s)ds-\int_t^T Z_s dW_s-\int_{{]t,T]}\times\Rd\setminus\{0\}}U_s(x)\tilde{N}(ds,dx) \,,
\end{align}
for each $t\in {[0,T]}$, where $W$ is a Brownian motion, $\tilde{N}$ is a compensated Poisson random measure  independend of $W$.
The BSDE \eqref{eq_lp_solution_informal} itself will be denoted by $(\xi,f)$.

\subsection{Related Works}\label{related_works}

For nonlinear BSDEs ($\xi$, $f$) driven by Brownian \changed{motion}, existence and uniqueness results  were first systematically studied by Pardoux and Peng \cite{pardoux1990adapted} with $(\omega,y,z)\mapsto f(\omega,y,z)$ Lipschitz in $(z,y)$ and $\xi$ square integrable. The importance of BSDEs in mathematical finance and stochastic optimal control was further elaborated by various works e.g. by El Karoui et al. \cite{elkaroui} \changed{which consider} Lipschitz generators, $L^p$-solutions and Malliavin derivatives of BSDEs in the Brownian setting. The ambition to weaken the assumptions on $f$ and $\xi$ to still guarantee a unique solution gave birth to a large number of contributions, where -- in the case of a generator with Lipschitz dependence on the $z$-variable --  at least a few should be mentioned herein:
Pardoux \cite{pardoux} and Briand and Carmona \cite{briand2000bsdes} considered monotonic generators w.r.t. $y$ with different growth conditions.
Mao \cite{mao} used the Bihari-LaSalle inequality to generalize the growth condition.
Briand et al. \cite{briand2003lp} proved existence and uniqueness of a solution in
the case where the generator may have a general growth in the $y$-variable and both $\int_0^T |f(s,0,0)|ds$ and $\xi$ belong to $L^p$ for some $p \ge 1$.
Generalizing the driving randomness,
Tang and Li \cite{tang1994necessary} and many other papers studied BSDEs including jumps by a Poisson random measure independent of the Brownian motion. 
Treating BSDEs in the case of quadratic growth in the $z$-variable, a considerable amount of articles was published in the recent years starting from the seminal paper of Kobylanski \cite{kobylanski2000} in 2000 to recent papers using BMO methods such as \cite{BriandRichou} in the Brownian case or also comparison theorems like in \cite{FujiiTakahashi} who consider an additional Poisson random measure as driving noise. We skip detailed comments in the direction of quadratic growth BSDEs as we will not consider this setting in our article.\bigskip

Recent and most relevant for the present paper are the results by
 Kruse and Popier \cite{Kruse} considering $\lp$-solutions for BSDEs \changed{driven by Brownian motion, a Poisson random measure and an additional martingale term under a monotonicity condition.} They included the case of random time horizons.
 Yao \cite{Yao} studied $\lp$-solutions to BSDEs with a finite activity \levy process for $1 < p < 2$ and used a generalization for the monotonicity assumption similar to the one of \cite{mao} and also used in Sow \cite{bamba}.
Generalizing the $\lp$-assumptions for the monotonic generator setting, in \cite{fan2019existence} the existence (and uniqueness in \cite{buckdahn2018uniqueness}) of a solution was proven for a scalar linearly growing BSDE when the terminal value $\xi$ admitted integrability of $|\xi|\exp\left(\mu \sqrt{2\log(1+|\xi|)}\right)$ for a parameter $\mu > \mu_0$, for some critical value $\mu_0 > 0$. \changed{ Moreover, a counterexample  in \cite{fan2019existence} shows that for the case $\mu < \mu_0$ the preceding integrability is not sufficient to guarantee existence. In the critical case $\mu=\mu_0$ they prove existence and uniqueness of a solution assuming a uniform Lipschitz generator.}

\subsection{Main Contribution}\label{amin_contribution}

Within our approach, the results shed new light to the extensive literature of BSDE existence and uniqueness results as follows. In \cite{Kruse17}, Kruse and Popier designed function spaces such that their results of \cite{Kruse} extend to $1<p<2$. In the present article, we show that the BSDEs' solutions for $1<p<2$ are even contained in the usual $L^p$ spaces as defined for $p\geq 2$.
Moreover an additional martingale term $M$ orthogonal to $W$ and $\tilde{N}$ as used by Kruse and Popier \cite{Kruse} can also be added to our setting as an extension of \eqref{eq_lp_solution_informal},
\begin{align}\label{eq:extended_bsde}
Y_t=\xi+\int_t^T f(s,Y_s,Z_s,U_s)ds-\int_t^T Z_s dW_s-\int_{{]t,T]}\times\R^\ell\setminus\{0\}}U_s(x)\tilde{N}(ds,dx) - \int_t^T dM_s ,
\end{align}

 with unknown variables $(Y,Z,U,M)$, as the careful analysis in their paper shows how the bracket process $[M]$ has to be treated in an a priori estimate. All the results we obtain are still valid in this extended setting - see \autoref{rem:extendedsetting} below. Nonetheless we decided to omit the presentation of the straightforward martingale term, since the main difficulty lies in the treatment of the compensated Poisson random measure. \bigskip

The paper of Geiss and Steinicke \cite{geiss2018monotonic}, placed in a 1-dimensional $L^2$-setting only, requires a linear growth condition on the generator and needs approximation results for the comparison theorem, while the present setting allows first of all general growth, but even uses a simpler approximation technique for the comparison theorem avoiding \changed{deep-lying} measurability results and, for $p\geq 2$, only requires comparison of the generators on the solution processes. \bigskip

Furthermore, in contrast to \cite{briand2003lp}, \cite{Kruse} and others, this article establishes the more general monotonicity (Osgood) condition with a non-decreasing, concave function $\rho$ to relax the generator's dependence on $y$ (see also Mao \cite{mao}). This includes e.g. continuities of the type as the function $y\mapsto -y\log(|y|)$ possesses at $y=0$. Using the general approach, similar a priori estimates are shown to still hold true in order to guarantee uniqueness of an $\lp$-solution, $p\geq 2$. \bigskip

In addition, the results of Yao \cite{Yao} are extended in the sense that we do not require the jump process to have a finite \levy measure.\bigskip

 Hence, we close several gaps in the theoretical understanding of solutions to BSDEs driven by a \levy process, for the class of generators which are Lipschitz in the $z$- and $u$-variables. 
 The more delicate techniques needed for this paper's approach for existence and uniqueness are inspired by the ideas of \cite{briand2003lp} along with \cite{Kruse} and \cite{elkaroui}. In that spirit, before starting the main proofs, we obtain useful a priori estimates for the solution processes. For the comparison theorem we enhance ideas and simplify proofs from \cite{geiss2018monotonic} and \cite{royer}.

\subsection{\changed{Structure of the Pape}r}\label{intro:structure}

This paper is organized in the following way:
First we establish the setting in \autoref{sec:setting} and state the assumptions and the main theorem (\autoref{sec:main_theorem}). After developing a priori estimates in \autoref{sec:apriori_estimates} we finally prove existence and uniqueness of $\lp$-solutions for $p > 1$  in \autoref{sec:proof_of_lp_existence} and end up with the comparison results for $p \ge 2$ and $1 < p < 2$ in \autoref{sec:comparison_result}.

\section{Setting} \label{sec:setting}

Throughout the paper, we will use the following setting:
In dimension $d \ge 1$, let $|\cdot|$ denote the Euclidean distance. For $x,y \in \Rd$ we write $\innerproduct{x}{y} = \sum_{i=1}^d x_i y_i$, and for $z \in \Rdk, k\geq 1,$ we denote $|z|^2 = \trace({z z^*})$.  The operations $\min(a,b)$ and $\max(a,b)$ will be denoted by $a \wedge b$ and $a \vee b$.\bigskip

 Let $X=\left(X_t\right)_{t\in{[0,T]}}$ be a c\`adl\`ag L\'evy process with values in $\Rd$ on a complete probability space $(\Omega,\mathcal{F},\mathbb{P})$
with L\'evy measure $\nu$. By
$\left({\mathcal{F}_t}\right)_{t\in{[0,T]}}$ we will denote the augmented natural filtration of $X$ and assume that $\mathcal{F}=\mathcal{F}_T.$ 
Equations or inequalities for objects on these spaces are considered up to $\mathbb{P}$-null sets. Conditional expectations $\E\left[\ \cdot\ \middle|\mathcal{F}_t\right]$ will be denoted by $\E_t$.\medskip

The L\'evy-It\^o decomposition of $X$ can be written as
\begin{equation}\label{LevyIto}
X_t = a t + \Sigma W_t   +  \int_{{]0,t]}\times \{ |x|\le1\}} x\tilde{N}(ds,dx) +  \int_{{]0,t]}\times \{ |x|> 1\}} x  N(ds,dx),
\end{equation}
where $a\in\Rd$, $\Sigma \in \Rdk$ with full column rank, $W$ is a $k$-dimensional standard Brownian motion and $N$ ($\tilde N$) is the (compensated) Poisson random measure corresponding to $X$. For the general theory of \levy processes, we refer to \cite{applebaum} or \cite{satou}. This setting can be adapted to a pure jump process, if one sets $\Sigma = 0$ and omits the stochastic integrals with respect to $W$ in the BSDE. 
Generalizing the above setting slightly, $\mathcal{F}$ can be assumed to be generated by a $k$-dimensional Brownian motion $W$ and an independent (from $W$) compensated Poisson random measure $\tilde{N}$ on $\mathbb{R}^\ell\!\setminus\! \{0\}$ for some $\ell\geq 1$, the assumption of a driving process $X$ can in principle be omitted. For convenience however, we will stick to the setting emerging from a driving L\'evy process $X$.


\subsection{Notation}\label{sec:notation}

Let $0<p\leq \infty$.
\begin{itemize}
	\item
	
	We use the notation $(L^p,\|\cdot\|_p):=\left(L^p(\Omega,\cF,\mathbb{P}),\|\cdot\|_{L^p}\right)$ for the space of all $\cF$-measurable functions $g : \Omega \to \Rd$ with
	$$
		||g||_{\lp} := \left(\int_{\Omega} |g|^p d\PP \right)^{1/p} < \infty \,
		\,\,\,
		 \text{if } p < \infty,
		\,\,\,
		\text{ and }
			\,\,\,
	||g||_{L^\infty} := \esssup_{\omega \in \Omega} |g(\omega)| < \infty .
	$$
	
\item Let  $\mathcal{S}^p$ denote the  space of all $\filtration$-progressively measurable and c\`adl\`ag processes \\ $Y\colon\Omega\times{[0,T]} \rightarrow \Rd$ such that
\equa
\left\|Y\right\|_{\mathcal{S}^p}:=\Big\|\sup_{0\leq t\leq T} \left|Y_{t}\right|\Big\|_p < \infty \,.
\tion

\item We define $L^p(W) $ as the space of all progressively measurable processes $Z\colon \Omega\times{[0,T]}\rightarrow \Rdk$  such that
\equa
\left\|Z\right\|_{L^p(W) }:=\left\|\left(\int_0^T\left|Z_s\right|^2 ds\right)^\frac{1}{2}\right\|_p < \infty \,.
\tion

\item Let $\Rzerod:= \Rd\!\setminus\!\{0\}$. We define $L^p(\tilde N)$ as the space of all random fields $U\colon \Omega\times{[0,T]}\times{\Rzerod}\rightarrow \Rd$ 
which are measurable with respect to
$\mathcal{P}\otimes\mathcal{B}(\Rzerod)$ (where $\mathcal{P}$ denotes the predictable $\sigma$-algebra on $\Omega\times[0,T]$ generated
by the left-continuous $\filtration$-adapted processes, and $\mathcal{B}$ is the Borel-$\sigma$-algebra) such that
\equa
\left\|U\right\|_{L^p(\tilde N) }:=\left\|\left(\int_0^T\int_{\Rzerod}\left|U_s(x)\right|^2 \nu(dx)ds\right)^\frac{1}{2}\right\|_p < \infty\,.
\tion

\item $L^2(\nu):= L^2(\Rzerod, \mathcal{B}(\Rzerod), \nu),$ $\|\cdot \|:=\|\cdot \|_{L^2(\nu)}.$

\item { $L^p([0,T]):=L^p([0,T],\mathcal{B}([0,T]), \lambda)$, where $\lambda$ is the Lebesgue measure on ${[0,T]}$.}

\item $L_{loc}(W)$ denotes the space of $\Rdk$-valued progressively measurable processes, such that for every $t > 0$, 
$$
\int_0^t | Z_s |^2 ds < \infty \,, \quad \mathbb{P}\text{-a.s.}
$$

\item $L_{loc}(\tilde{N})$ denotes the space of $\mathcal{P}\otimes\mathcal{B}(\Rzerod)$-measurable random fields $U\colon \Omega\times{[0,T]}\times{\Rzerod}\rightarrow \Rd$, such that for every $t > 0$, 
$$
\int_0^t\int_{\Rzerod} \left(| U_s(x) |^2\vee|U_s(x)|\right) \nu(dx)ds < \infty \,, \quad \mathbb{P}\text{-a.s.}
$$

\item With a slight abuse of notation we define
\equa
&& \hspace{-2em} L^p(\Omega; L^1([0,T])) \\ \hspace{-1em}&:=&\!\!\!\!\! \left  \{F : \Omega \times [0, T] \to \R: F \text{ is } \cF \otimes \mathcal{B}([0,T])\text{-measurable, } { \left\| \!\int_0^T \!\!|F(\omega, t)| dt\right\|_p }< \infty  \right \}.  \notag
\tion

For $F \in  L^p(\Omega; L_1([0,T]))$ we define 
\begin{align}
\label{eq:i_f_and_k_f} I_F(\omega):= \int_0^T F(\omega, t)dt \quad \text{ and } \quad K_F(\omega, s) := 
\Bigg\{\begin{array}{lr}
\frac{F(\omega, s)}{I_F(\omega)}, & \text{if } I_F(\omega) \ne 0\\
0, & \text{if } I_F(\omega) = 0
\end{array}.
\end{align}
\changed{
The notions $I_F$ and $K_F$ are designed to make use of the simple properties, $F = I_F K_F$ and $I_F^{p-1} \int_0^T F(t) dt = I_F^p$ together with $\int_0^T K_F dt = 1$, $\mathbb{P}$-a.s. These properties are used e.g. below equation \eqref{gammaeq2}.
}

\item

We consider the terminal condition $\xi$ to be an $\cF_T$-measurable random variable and the generator to be a random function $f: \Omega \times [0,T] \times \Rd \times \Rdk \times L^2(\nu) \to \Rd$.
\end{itemize}

\begin{definition}\label{defi:lloc-solution}
 An {\bf $\lloc$-solution to a BSDE} $(\xi,f)$ with terminal condition $\xi$ and generator $f$ is a triplet 
$$	
(Y,Z,U)\in L_{loc}(W)\times L_{loc}(W)\times L_{loc}(\tilde N), 
$$
  adapted to $\left({\mathcal{F}_t}\right)_{t\in{[0,T]}}$, which satisfies for all $t\in{[0,T]}$,
\begin{align*}
Y_t=\xi+\int_t^T f(s,Y_s,Z_s,U_s)ds-\int_t^T Z_s dW_s-\int_{{]t,T]}\times\Rzerod}U_s(x)\tilde{N}(ds,dx),\quad \mathbb{P}\text{-a.s}.
\end{align*}
\end{definition}

\begin{definition}\label{defi:lp-solution}
An {\bf $L^p$-solution to a BSDE} $(\xi,f)$ with terminal condition $\xi$ and generator $f$ is an $\lloc$-solution $(Y,Z,U)$ to the BSDE $(\xi,f)$ which satisfies  
\begin{align*}
(Y,Z,U)\in \mathcal{S}^p\times L^p(W)\times L^p(\tilde N).
\end{align*}
\end{definition}

\begin{remark}\label{rem:extendedsetting}
	An extension of our setting in the way of \cite{Kruse17} is the following:
	
	Redefine the spaces above using a filtration $\left({\mathcal{F}_t}\right)_{t\in{[0,T]}}$ on $(\Omega, \F, \PP)$, that is assumed to be quasi-left continuous, satisfies the usual conditions and supports a $k$-dimensional Brownian motion $W$ and a compensated Poisson random measure $\tilde{N}$ on $\R^\ell \!\setminus\! \{0\}$. 
	Furthermore, we introduce the space $\mloc$ of c\`adl\`ag local martingales orthogonal to $W$ and $\tilde{N}$ and the space $\mathcal{M}$ of true martingale processes in $\mloc$. Moreover define 
	$$
	\mathcal{M}^p := \left\{ M \in \mathcal{M} : \E\left[ ([M]_T)^{p/2} \right] < \infty \right\}
	$$
	 and in the sense of \autoref{defi:lloc-solution} and \ref{defi:lp-solution} let an $\lloc$-solution (respectively $\lp$-solution) to a BSDE \eqref{eq:extended_bsde} be a tuple
	$$	
	(Y,Z,U,M)\in L_{loc}(W)\times L_{loc}(W)\times L_{loc}(\tilde N) \times \mloc, 
	$$
	respectively
	$$
	(Y,Z,U,M)\in \mathcal{S}^p\times L^p(W)\times L^p(\tilde N) \times \mathcal{M}^p
	$$
	safisfying equation \eqref{eq:extended_bsde} instead of \eqref{eq_lp_solution_informal}. As mentioned in the introduction, our main results in the following sections can also be shown within the extended setting producing some extra lines of technical computations.
\end{remark}


\subsection{\levy process with finite measure}\label{sec:finlev}

The driving \levy process, given by its \levy-\ito-decomposition  
\eqref{LevyIto}
will be approximated for $n\geq 1$ by
$$X^n_t=at+\Sigma W_t+\int_{]0,t]\times \{|x|>1\}}x N(ds,dx)+\int_{]0,t]\times \{1/n \leq |x|\leq 1\}}x \tilde{N}(ds,dx).$$
The process $X^n$ has a finite L\'evy measure. Note furthermore, that the compensated Poisson random measure associated to $X^n$ can be expressed as $\tilde{N}^n=\chi_{\{1/n \leq |x|\}}\tilde{N}$, where $\chi_A$ denotes the indicator function of a set $A$. Let 
\begin{align*} 
\mathcal{F}^0&:=\{\Omega, \emptyset\} \vee \mathcal{N},  \nonumber \\
\mathcal{F}^n&:= \sigma(X^n)  \vee \mathcal{N}, \quad n\geq 1,
\end{align*}
where $\mathcal{N}$ stands for  the null sets of $\mathcal{F}.$ 
Denote by  $\E_n$ the conditional 
expectation  $\E\left[\ \cdot\ \middle|\mathcal{F}^n\right]$.  

\section{Main Theorem} \label{sec:main_theorem}


With this setting in mind, we now state the main theorem based on the following assumptions, with a slight distinction for $p \ge 2$ and $p < 2$, which turns out to be quite natural for the proofs. Instead of a Lipschitz condition, we require the weaker conditions \ref{A3ge2} and respectively \ref{A3le2}, referred to as one-sided Lipschitz or monotonicity condition for the generator $f$.


\subsection{Assumptions}\label{sec:assumptions}

\begin{enumerate}[label={(A\,\arabic*)}]

	\item\label{A1} For all $(y,z,u) \in \Rd \times \Rdk \times L^2(\nu): (\omega,s)\mapsto f(\omega,s,y,z,u)$ is progressively measurable and the process $\fzero=(f(t,0,0,0))_{t\in{[0,T]}}$ is in $ L^p(\Omega;L^1([0,T]))$.

	\item\label{A2}  For all $r>0$ there are nonnegative, progressively measurable processes $\Phi$, $\psi_r$ with 
	$$
	\left \|\int_0^T \Phi(\cdot,s)^2ds \right \|_{\infty}<\infty
	$$
	and   $\psi_r\in L^1(\Omega\times{[0,T]})$ such that for all $(z,u) \in \Rdk \times L^2(\nu)$,
	\begin{align*}
		&\sup_{|y|\leq r}|f(t,y,z,u)-\fzerot|\leq \psi_r(t)+\Phi(t)(|z|+\|u\|), \quad \mathbb{P}\otimes\lambda\text{-a.e.}
	\end{align*}
	
	\item[\namedlabel{A3ge2}{(A3$_{\ge 2}$)}] For $p\geq 2$: 
	
	For $\lambda$-almost all $s$, the mapping $(y,z,u)\mapsto f(s,y,z,u)$ is $\mathbb{P}$-a.s. continuous. Moreover, there is a nonnegative function $\alpha\in L^1([0,T])$ and progressively measurable processes $\mu, \beta$ with \sloppy $\int_0^T\left(\mu(\omega,s)+\beta(\omega,s)^2\right) ds < \infty$, $\mathbb{P}$-a.s. such that for all $(y,z,u), (y',z',u') \in \Rd \times \Rdk \times L^2(\nu)$,
	\begin{align}\label{eq:A3ge2inequality}
	&|y-y'|^{p-2} \innerproduct{y-y'}{f(t,y,z,u)-f(t,y',z',u')} \nonumber \\
	&\leq \alpha(t)|y-y'|^{p-2}\rho(|y-y'|^2)+\mu(t)|y-y'|^p+\beta(t)|y-y'|^{p-1}(|z-z'|+\|u-u'\|),
	\end{align}
		
	$\mathbb{P}\otimes\lambda\text{-a.e.}$, for a nondecreasing, continuous and concave function $\rho$ from ${[0,\infty[}$ to itself, satisfying  $\rho(0)=0$,  
	$\lim_{x\to0} \frac{\rho(x^2)}{x} = 0$ and the Osgood condition $\int_{0^+}\frac{1}{\rho(x)}dx := \int_{0}^\epsilon \frac{1}{\rho(x)}dx=\infty $, for some $\epsilon > 0$. 

	\item[\namedlabel{A3le2}{(A3$_{< 2}$)}] For $0<p<2$:
	
	 For $\lambda$-almost all $s$, the mapping $(y,z,u)\mapsto f(s,y,z,u)$ is $\mathbb{P}$-a.s. continuous. Moreover, there is a nonnegative function $\alpha\in L^1([0,T])$, $C>0$ and  progressively measurable processes $\mu, \beta_1, \beta_2$ with $\int_0^T\left(\mu(\omega,s)+\beta_1(\omega,s)^2+\beta_2(\omega,s)^q\right) ds < C$, $\mathbb{P}$-a.s for some $q > 2$. such that for all $(y,z,u), (y',z',u') \in \Rd \times \Rdk \times L^2(\nu)$, $y\neq y'$	
	\begin{align}\label{eq:A3le2inequality}
	&|y-y'|^{p-2}\innerproduct{y-y'}{f(t,y,z,u)-f(t,y',z',u')} \nonumber \\
	&\leq \alpha(t)\rho(|y-y'|^p)+\mu(t)|y-y'|^p+|y-y'|^{p-1}\left(\beta_1(t)|z-z'|+\beta_2(t)\|u-u'\|\right),
	\end{align}
	
	$\mathbb{P}\otimes\lambda\text{-a.e.}$, for a nondecreasing, continuous and concave function $\rho$ from ${[0,\infty[}$ to itself, satisfying $\rho(0)=0$, 
	$\lim_{x \to 0} \frac{\rho(x^p)}{x^{p-1}} = 0$ and $\int_{0^+}\frac{1}{\rho(x)}dx=\infty$.

\end{enumerate}

\begin{remark}~ 
	\label{remA3}
	\begin{enumerate}[label=({\roman*})]
		\item
		 The limit assumptions $\lim_{x \to 0} \frac{\rho(x^2)}{x} = 0$ together with \eqref{eq:A3ge2inequality} or $\lim_{x \to 0} \frac{\rho(x^p)}{x^{p-1}} = 0$ together with \eqref{eq:A3le2inequality} already imply that the generator $f$ is Lipschitz in $z, u$. Moreover $\beta$ (and analogously for \ref{A3le2} the process $\beta_1 + \beta_2$) can take the role of $\Phi$ in \ref{A2}. Nonetheless, for convenience in the proofs, we will still use the generic function $\Phi$.

		\item\label{remA3:A3bounds}
		The $\rho$-function appearing in the right hand sides of \ref{A3ge2} and \ref{A3le2} admits the following inequalities, which play important roles in the proofs:
		\changed{
		\begin{enumerate}[label=({\alph*})]
			\item\label{remA3ge2} 
			$
			\rho(|y|^2)|y|^{p-2} \le \rho(|y|^p) + \rho(1) |y|^p \,,
			\quad \text{for } p \ge 2,
			$
			\item\label{remA3le2} 
			$
			\rho(|y|^p)|y|^{2-p} \le \rho(|y|^2) + \rho(1) |y|^2 \,,
			\quad \text{for }  0 < p < 2.
			$
		\end{enumerate}		
	}
	\end{enumerate}
\end{remark}

\begin{proof}
 For \ref{remA3:A3bounds}, we see that, if $|y| < 1$, then $|y|^{p-2} < 1$ and by the concavity of $\rho$,
	$$
		\rho(|y|^2)|y|^{p-2} \le \rho(|y|^2|y|^{p-2}) = \rho(|y|^p)\,.
	$$
	For $|y| \ge 1$ we have by the concavity of $\rho$,
	$$
		\rho(|y|^2)|y|^{p-2} \le \rho(1)|y|^2|y|^{p-2} = \rho(1)|y|^{p} \,.
	$$
	
	The case $ 0 < p \le 2$ is similar. 
\end{proof}

\begin{remark}~ 
	\begin{enumerate}[label=({\roman*})]
		\item
		In \ref{A3le2}, if $\beta_2$ is deterministic, we could impose the weaker condition $q=2$ as described later in \autoref{rem:beta_2_deterministic}.
		
		\item
			The following example is constructed in order to demonstrate the possibilities in this setting for $d=1, p>1$. All the involved expressions are chosen to exploit the assumptions on the coefficients which may be time-dependent, unbounded and some even random. The generator's dependence on $y$ is not Lipschitz (not even one-sided Lipschitz) and of super-linear growth:			
			\begin{align*}
			f(\omega,t,y,z,u)=&\frac{-1}{\sqrt{t}}y\log(|y|)-\mu(\omega,t)\left(y^3+y^\frac{1}{3}\right)+\beta_1(\omega,t)(z+\sin(z)\cos(y))\\
			&\ +\beta_2(\omega,t)\int_{\R_0}\left(\arctan(y\kappa(x)u(x))+u(x)\right)\kappa(x)\nu(dx)+f_0(t),
			\end{align*}
			where 
			\begin{itemize}
			\item $\mu$ is given by $\mu(\omega,t)=\sum_{n=1}^\infty \frac{1}{n^2\sqrt{t-t_n(\omega)}}$, with $(t_n(\omega))_{n\geq 1}$ being a numeration of the jumps of the trajectory $t\mapsto X_t(\omega)$ of the L\'evy process and $\mu(t,\omega)=0$ if $t\mapsto X_{t}(\omega)$ has no jumps,
			\item $\beta_1(\omega,t)=\begin{cases}\frac{\chi_{[T/2,T]}(t)}{\sqrt{|t-W_{T/2}(\omega)|}\left(|\log(|t-W_{T/2}(\omega)|+1)|\right)},& \text{when defined,}\\ \quad\quad\quad\quad 0 & \text{else,}\end{cases}$
			\item $\beta_2(\omega,t)=\begin{cases}\frac{\chi_{[T/3,T]}(t)}{\left|\log\left(\left|t-\frac{|W_{T/3}(\omega)|}{1+|W_{T/3}(\omega)|}\right|\right)\right|},& \text{when defined,}\\ \quad\quad\quad\quad 0 & \text{else,}\end{cases}$
			\item $\kappa(x)=1\wedge |x|$,
			\item $f_0(\omega,t)=\begin{cases} \left(\int_0^t\frac{\sqrt{s}\ \exp\left(\frac{W_s^2}{2s}\right)}{|W_s|\left(|\log\left(\frac{|W_s|}{\sqrt{t}}\right)|+1\right)^2}ds\right)^\frac{1-p}{p}\frac{\sqrt{t}\ \exp\left(\frac{W_t^2}{2t} \right)}{|W_t|\left(|\log\left(\frac{|W_t|}{\sqrt{t}}\right)|+1\right)^2}, & \text{when defined,}\\ \quad\quad\quad\quad 0 & \text{else.}\end{cases}$
	\end{itemize}
	\end{enumerate}
	
\end{remark}

\subsection{Main Theorem}\label{subsec:main_theorem}
\begin{theorem}[Existence and Uniqueness]\label{existence}
	Assume that the terminal condition $\xi$ is in $L^p$ and the generator $f$ satisfies
	\ref{A1},\ref{A2},\ref{A3ge2}, for $p\geq 2$ or \ref{A1}, \ref{A2}, \ref{A3le2}, for $1<p<2$,
	then there exists a unique $\lp$-solution to the BSDE  $(\xi,f)$.
\end{theorem}

We will prove this theorem in \autoref{sec:proof_of_lp_existence} after presenting necessary a priori estimates in the next section.

\section{A Priori Estimates and Stability} \label{sec:apriori_estimates}

Throughout the next sections, recall that $\fzerot = f(t,0,0,0)$, and that $\Ifzero$ and $\Kfzero$ are defined as in \eqref{eq:i_f_and_k_f}. 

\begin{remark}\label{rem:weaker_assumption_on_monotonicity}
	\changed{For the results in this section, it suffices to require a weaker condition than \ref{A3ge2}. We define this adapted assumption \ref{a3le2} with the same requirements, except one: We replace the monotonicity condition \eqref{eq:A3ge2inequality} by

	\begin{itemize}
	\item[\namedlabel{a3ge2}{(a3$_{\ge 2}$)}]\begin{align}
	\label{rem:weaker_assumption_on_monotonicity_1}
	&|y|^{p-2} \innerproduct{y}{f(t,y,z,u)} \nonumber \\
	&\quad\leq \alpha(t)|y|^{p-2}\rho(|y|^2)+\mu(t)|y|^p+\beta(t)|y|^{p-1}(|z|+\|u\|)+|y|^{p-1}|f_0(t)|,
	\end{align}
\end{itemize}
and analogously we adapt the assumption \ref{A3le2} to derive the weaker \ref{a3le2} by replacing inequality \eqref{eq:A3le2inequality},
\begin{itemize}
\item[\namedlabel{a3le2}{(a3$_{< 2}$)}]
\begin{align}
	\label{rem:weaker_assumption_on_monotonicity_2}
	&|y|^{p-2} \innerproduct{y}{f(t,y,z,u)} \nonumber \\
	&\quad\leq \alpha(t)\rho(|y|^p)+\mu(t)|y|^p+|y|^{p-1}\left(\beta_1(t)|z|+\beta_2(t)\|u\|\right)+|y|^{p-1}|f_0(t)|,
	\end{align}
	$\mathbb{P}\otimes\lambda\text{-a.e.}$ for all $(y,z,u) \in \Rd \times \Rdk \times L^2(\nu)$.
\end{itemize} }
\end{remark}
		
\begin{lemma}\label{lem_unif_int}
		
		 
	If a sequence of random variables $\sequence{V}$ in $L^p$ satisfies $\lim_{n\to \infty}\E|V_n|^p=0$, then for a function $\rho$ as in the assumptions, we have
	$$\lim_{n\to\infty}\E \left[ \rho\left(|V_n|^2\right)^{\frac{p}{2}}\right] = 0.$$ 
	
\end{lemma}
	\begin{proof}
	This follows from the continuity of $\rho$, $\rho(0) = 0$ and the uniform integrability of $(|V_n|^p)_{n\geq 1}$, 
	\begin{align*}
		  \rho\left(|V_n|^2\right)^{\frac{p}{2}}
		\le   \left(a + b |V_n|^2\right)^{\frac{p}{2}}
		\le 2^{\frac{p}{2}-1} \left( a^{\frac{p}{2}} + b^{\frac{p}{2}} |V_n|^p\right)
		\,,
	\end{align*}
	since $\rho(x)\leq a + bx$ for some $a,b>0$ and the above inequality shows that also $(\rho(|V_n|^2)^{\frac{p}{2}})_{n\geq 1}$ is a uniformly integrable sequence.
\end{proof}

The following two propositions show that the norms of the $Z$ and $U$ processes can be controlled by expressions in $Y$ and $f_0$. Note that the bounds in \autoref{YdeterminesZU} and \autoref{YdeterminesZUple2} differ slightly, so that the application of \autoref{YdeterminesZU} in \autoref{sec:proof_of_lp_existence} needs the assertion of \autoref{lem_unif_int}.

\begin{proposition}\label{YdeterminesZU}
	Let $p \ge 2$ and let $(Y,Z,U)$ be an $\lloc$-solution to the BSDE $(\xi,f)$. If $\xi\in L^p$, $Y \in \mathcal{S}^p$ and \ref{A1} and \changed{\ref{a3ge2}} are satisfied, then $(Y,Z,U)$ is an $\lp$-solution. \bigskip
	
	More precisely, there is a constant $C > 0$ depending on $p,T,\alpha,\mu,\beta$ such that for all $t\in {[0,T]}$,
	\begin{align*}
	&\E\left[\left(\int_t^{T}|Z_s|^2ds\right)^\frac{p}{2}\right]+\E\left[\left(\int_t^{T}\|U_s\|^2 ds\right)^\frac{p}{2}\right]\\
	&\leq C\left(\E\left[\sup_{s\in{[t,T]}}|Y_s|^p\right]+
	\E\left[ \rho\left(\sup_{s\in{[t,T]}}|Y_s|^2\right)^{\frac{p}{2}}\right]+\E \left[\left(\int_t^T|\fzero(s)|ds\right)^p\right]\right).
	\end{align*}
	
\end{proposition}

\begin{proof}
	This proof generalizes the arguments in \cite[Lemma 3.1]{briand2003lp}.
	
{\bf Step 1:}\\	
	For $t\in{[0,T]}$ and $n \geq 1$ define the stopping times
	\begin{align*}
	\tau_n:=\inf\left\{s\in [t,T]: \int_t^{s}|Z_s|^2ds\geq n\right\}\wedge \inf\left\{s\in{[t,T]}: \int_t^{s}\|U_s\|^2 ds\geq n\right\}.
	\end{align*}
	It\^o's formula implies
	\changed{
	\begin{align}\label{eq:ito_after_stopping_time}
	&|Y_t|^2+\int_t^{\tau_n}|Z_s|^2ds+\int_t^{\tau_n}\|U_s\|^2 ds=|Y_{\tau_n}|^2+2\int_t^{\tau_n}
	\tempinnerproduct{Y_s}{f(s,Y_s,Z_s,U_s)}ds \nonumber\\
	&-2\int_t^{\tau_n} 
	\tempinnerproduct{Y_s}{Z_sdW_s}
	-\int_{{]t,\tau_n]}\times\Rzerod}\left(|Y_{s-}+U_s(x)|^2-|Y_{s-}|^2\right)\tilde{N}(ds,dx),	
	\end{align}
}
	from which we infer by \changed{\ref{a3ge2}} that
	\begin{align*}
	&\int_t^{\tau_n}|Z_s|^2ds+\int_t^{\tau_n}\|U_s\|^2 ds\leq|Y_{\tau_n}|^2+2\int_t^{\tau_n}\left(\alpha(s)\rho(|Y_s|^2)+\mu(s)|Y_s|^2\right)ds\\
	&+\int_t^{\tau_n}\beta(s)|Y_s|(|Z_s|+\|U_s\|)ds+2\int_t^{\tau_n}|Y_s||\fzero(s)| ds\\
	&+2\left|\int_t^{\tau_n} \tempinnerproduct{Y_s}{Z_sdW_s}\right|+\left|\int_{{]t,\tau_n]}\times\Rzerod}\left(|Y_{s-}+U_s(x)|^2-|Y_{s-}|^2\right)\tilde{N}(ds,dx)\right|.
	\end{align*}
	Taking the power $\frac{p}{2}$, we find a constant $\czero > 0$ such that
	\begin{align*}
	&\left[\int_t^{\tau_n}|Z_s|^2ds\right]^\frac{p}{2}+\left[\int_t^{\tau_n}\|U_s\|^2 ds\right]^\frac{p}{2}\leq \czero\biggl(|Y_{\tau_n}|^p+\left[\int_t^{\tau_n}\left(\alpha(s)\rho(|Y_s|^2)+\mu(s)|Y_s|^2\right)ds\right]^\frac{p}{2}\\
	&+\left[\int_t^{\tau_n}\beta(s)|Y_s|(|Z_s|+\|U_s\|)ds\right]^\frac{p}{2}+\left[\int_t^{\tau_n}|Y_s||\fzero(s)|ds\right]^\frac{p}{2}\\
	&+\left|\int_t^{\tau_n} 
	\tempinnerproduct{Y_s}{Z_sdW_s}
	\right|^\frac{p}{2}+\left|\int_{{]t,\tau_n]}\times\Rzerod}\left(|Y_{s-}+U_s(x)|^2-|Y_{s-}|^2\right)\tilde{N}(ds,dx)\right|^\frac{p}{2}\biggr).
	\end{align*}
	We continue our estimate (with another constant $c_1 > 0$)	
	\begin{align}\label{eq:propBeforeCases}
	&\left[\int_t^{\tau_n}|Z_s|^2ds\right]^\frac{p}{2}+\left[\int_t^{\tau_n}\|U_s\|^2 ds\right]^\frac{p}{2} \nonumber\\
	&\leq c_1\Biggl(\sup_{s\in{[t,T]}}|Y_s|^p+\left[\int_t^T  \alpha(s)ds \,   \right]^\frac{p}{2}\rho\left(\sup_{s\in [t,T]}|Y_s|^{2}\right)^{\frac{p}{2}}
	+\left[\int_t^T\mu(s)ds\right]^\frac{p}{2}\sup_{s\in{[t,T]}}|Y_s|^p \nonumber\\
	&+\left[\int_t^{\tau_n}\beta(s)|Y_s|(|Z_s|+\|U_s\|)ds\right]^\frac{p}{2}+\left[\int_t^{\tau_n}|\fzero(s)|ds\right]^\frac{p}{2}\sup_{s\in{[t,T]}}|Y_s|^\frac{p}{2} \nonumber\\
	&+\left|\int_t^{\tau_n} 
	\tempinnerproduct{Y_s}{Z_sdW_s}
	\right|^\frac{p}{2}+\left|\int_{{]t,\tau_n]}\times\Rzerod}\left(|Y_{s-}+U_s(x)|^2-|Y_{s-}|^2\right)\tilde{N}(ds,dx)\right|^\frac{p}{2}\Biggr).
	\end{align}
	
	To estimate the above further, we have to split up the range of values of $p\geq 2$.\medskip
	
	\textbf{Case 1:} $2 \le p \le 4$

	We use the following inequality given e.g. in \cite[Theorem 3.2]{MarinRoeck}, which states that for a local martingale $M$, given by $M(t)=\int_{{]0,t]}\times\Rzerod} g_s(x)\tilde{N}(ds,dx), t\in [0,T]$, there exists $\ctwo > 0$ such that the following inequality holds for $p'\in {]0,2]}$:
	$$\E\sup_{t\in [0,T]}\left|M_t\right|^{p'}\leq \ctwo\E\left[\left(\int_0^T\int_{\Rzerod} |g_s|^2\nu(dx)ds\right)^\frac{p'}{2}\right].$$
	Here, we will apply this inequality for $p' = p/2$ to the martingale 
	$$s\mapsto \int_{{]t,s\wedge \tau_n]}\times\Rzerod}\left(|Y_{s-}+U_s(x)|^2-|Y_{s-}|^2\right)\tilde{N}(ds,dx).$$
	Note that we can estimate the square of the above integrand $\mathbb{P}\otimes\lambda\otimes\nu$-a.e. by
	\begin{align}\label{eq:bound_U_by_2_Y}
	\left(|Y_{s-}+ U_s(x)|+|Y_{s-}|\right)^2 \left(|Y_{s-}+ U_s(x)|-|Y_{s-}|\right)^2  \le 16 \sup_{r \in [t,T]} |Y_r|^2  \,|U_s(x)|^2, 
	\end{align}
	since for all $s \in [t,T]$ we can bound the jump sizes 
	$|U_s(x)|$ by $2 \sup_{r \in [t,T]} |Y_r|$, $\mathbb{P}\otimes\lambda\otimes\nu$-a.e. (see \cite[Corollary 1]{morlais2009utility}) and since $\big||Y_{s-}+ U_s(x)|-|Y_{s-}|\big| \le |U_s(x)|$.
	
	We take suprema and expectations to get a constant $\cthree > 0
	$ such that
	\begin{align*}
	&\E\left[\int_t^{\tau_n}|Z_s|^2ds\right]^\frac{p}{2}+\E\left[\int_t^{\tau_n}\|U_s\|^2 ds\right]^\frac{p}{2}\\
	&\leq \cthree\Biggl(\E\left[\sup_{s\in{[t,T]}}|Y_s|^p\right]+\E\left[\rho\left(\sup_{s\in [t,T]}|Y_s|^{2}\right)^{\frac{p}{2}}\right]\\
	&+\E\left[\int_t^{\tau_n}\beta(s)|Y_s|(|Z_s|+\|U_s\|)ds\right]^\frac{p}{2}+\E\left[\left(\int_t^{\tau_n}|\fzero(s)|ds\right)^\frac{p}{2}\sup_{s\in{[t,T]}}|Y_s|^\frac{p}{2}\right]\\
	&+\E\left[\int_t^{\tau_n}|Y_s|^2|Z_s|^2ds\right]^\frac{p}{4}+\E\left[\int_t^{\tau_n}\sup_{r\in{[t,T]}}|Y_r|^2\|U_s\|^2ds\right]^\frac{p}{4}\Biggr).
	\end{align*}
	
	Young's inequality (see \autoref{young_inequality} in the Appendix) now gives us for an arbitrary $R>0$,
	\begin{align*}
	&\E\left[\int_t^{\tau_n}|Z_s|^2ds\right]^\frac{p}{2}+\E\left[\int_t^{\tau_n}\|U_s\|^2 ds\right]^\frac{p}{2}\\
	&\leq \cthree\Biggl(\E\left[\sup_{s\in{[t,T]}}|Y_s|^p\right]+\E\left[\rho\left(\sup_{s\in [t,T]}|Y_s|^{2}\right)^{\frac{p}{2}}\right]\\
	&+\E\left[\int_t^{\tau_n}\frac{R}{2}\beta(s)^2|Y_s|^2ds\right]^\frac{p}{2}+\frac{1}{(2R)^\frac{p}{2}}\left(\E\left[\int_t^{\tau_n}|Z_s|^2ds\right]^\frac{p}{2}+\E\left[\int_t^{\tau_n}\|U_s\|^2ds\right]^\frac{p}{2}\right)\\
	&+\frac{1}{2}\E\left[\int_t^{\tau_n}|\fzero(s)|ds\right]^p+\frac{1}{2}\E\sup_{s\in{[t,T]}}|Y_s|^p\\
	&+2\left(\frac{R}{2}\right)^\frac{p}{2}\E\sup_{s\in{[t,T]}}|Y_s|^p+\frac{1}{(2R)^\frac{p}{2}}\left(\E\left[\int_t^{\tau_n}|Z_s|^2ds\right]^\frac{p}{2}+\E\left[\int_t^{\tau_n}\|U_s\|^2ds\right]^\frac{p}{2}\right)\Biggr).
	\end{align*}
	Choosing now $R$ such that $\frac{2 \cthree}{(2R)^\frac{p}{2}}<1$ yields
	a constant $C > 0
	$ such that
	
	\begin{align*}
	&\E\left[\int_t^{\tau_n}|Z_s|^2ds\right]^\frac{p}{2}+\E\left[\int_t^{\tau_n}\|U_s\|^2 ds\right]^\frac{p}{2}\\
	&\leq C\left(\E\left[\sup_{s\in{[t,T]}}|Y_s|^p\right]+\E\left[\rho\left(\sup_{s\in [t,T]}|Y_s|^{2}\right)^{\frac{p}{2}}\right]+\E \left[\int_t^T|\fzero(s)|ds\right]^p\right).
	\end{align*}
	Taking the limit for $n\to\infty$ shows the assertion for $2 \le p \le 4$.\medskip

	\textbf{Case 2:} $p > 4$\\
	We start from \eqref{eq:propBeforeCases} following the same lines of the previous case. In this case the only difference is:
	\cite[Theorem 3.2]{MarinRoeck} states that for a local martingale $M$, given by $M(t)=\int_{{]0,t]}\times\Rzerod} g_s(x)\tilde{N}(ds,dx)$, $t\in [0,T]$ there exists $\cfour > 0$ such that the following inequality holds for all $p' \ge 2$:
	\begin{align}
	\label{eq:propCase2}
	\E\sup_{t\in [0,T]}\left|M_t\right|^{p'}\leq \cfour\E\left(\left[\int_0^T\int_{\Rzerod} |g_s|^2\nu(dx)ds\right]^\frac{p'}{2} + \int_0^T\int_{\Rzerod} |g_s|^{p'}\nu(dx)ds\right).
 	\end{align}
	For $p' = \frac{p}{2}$, we apply this inequality to the local martingale 
	$$s\mapsto \int_{{]t,s\wedge \tau_n]}\times\Rzerod}\left(|Y_{s-}+U_s(x)|^2-|Y_{s-}|^2\right)\tilde{N}(ds,dx).$$
	
	The first summand of \eqref{eq:propCase2} can be treated as in case 1. We focus on the second term which equals
	\begin{align}\label{eq:prop4squares}
	&\int_t^T\int_{\Rzerod} \left(|Y_{s-}+U_s(x)|^2-|Y_{s-}|^2\right)^{\frac{p}{2}}\nu(dx)ds\nonumber \\
	=& \int_t^T\int_{\Rzerod} \left(|Y_{s-}+U_s(x)|^2-|Y_{s-}|^2\right)^{\frac{p}{2}-2}
	\left(|Y_{s-}+U_s(x)|^2-|Y_{s-}|^2\right)^{2}\nu(dx)ds .	
	\end{align}
	
We can bound the integrands (as explained in \eqref{eq:bound_U_by_2_Y}) by
$$\big(|Y_{s-}+ U_s(x)|+|Y_{s-}|\big) \big(|Y_{s-}+ U_s(x)|-|Y_{s-}|\big)  \le 16 \sup_{r \in [t,T]} |Y_r|^2,$$ 
	and 
	$$\big(|Y_{s-}+ U_s(x)|+|Y_{s-}|\big) \big(|Y_{s-}+ U_s(x)|-|Y_{s-}|\big)  \le 4 \sup_{r \in [t,T]} |Y_r|  \,|U_s(x)|. $$ 
Hence we find a constant $\cfive > 0$, such that \eqref{eq:prop4squares} is smaller than
\changed{
$$
\cfive \int_t^T\int_{\Rzerod} \sup_{r \in [t,T]} |Y_r|^{p-2} |U_s(x)|^{2} \nu(dx)ds 
.
$$

Using Young's inequality for the conjugate couple $(\frac{p}{2}, \frac{p}{p-2})$, we have for arbitrary $R_1 > 0$,}
\begin{align*}
& \int_t^T\int_{\Rzerod} \sup_{r \in [t,T]} |Y_r|^{p-2} |U_s(x)|^{2} \nu(dx)ds 
=  \sup_{r \in [t,T]} |Y_r|^{p-2} \int_t^T\int_{\Rzerod}  |U_s(x)|^{2} \nu(dx)ds \\
\le&
 \left( \frac{p-2}{p} R_1^{\frac{p}{p-2}} \sup_{r \in [t,T]} |Y_r|^{p} + \frac{2}{pR_1^{\frac{p}{2}}}\left[\int_t^T\int_{\Rzerod}  |U_s(x)|^{2} \nu(dx)ds\right]^{\frac{p}{2}}\right).
\end{align*}
From here, similar steps as in case 1 conclude the proof.
\end{proof}

\begin{proposition}\label{YdeterminesZUple2}
	Let $0 < p < 2$ and let $(Y,Z,U)$ be an $\lloc$-solution to the BSDE $(\xi,f)$. If $\xi\in L^p$, $Y \in \mathcal{S}^p$ and \ref{A1} and \changed{\ref{a3le2}} are satisfied, then $(Y,Z,U)$ is an $\lp$-solution. \bigskip 
	
	 More precisely, there is a constant $C$ depending on $p,T,\alpha,\rho(1),\mu,\beta_1,\beta_2$ such that for all $t\in {[0,T]}$,
	\begin{align*}
	&\E\left[\int_t^{T}|Z_s|^2ds\right]^\frac{p}{2}+\E\left[\int_t^{T}\|U_s\|^2 ds\right]^\frac{p}{2}\\
	&\leq C\left(\E\left[\sup_{s\in{[t,T]}}|Y_s|^p\right]+
	\E\left[ \int_t^T \alpha(s) \rho\left(|Y_s|^p\right)ds\right]+\E \left[\int_t^T|\fzero(s)|ds\right]^p\right).
	\end{align*}
	The assertion holds true even if $q=2$ in \changed{\ref{a3le2}} since we do not use a higher integrability condition in the proof. 
\end{proposition}

\begin{proof}
	\changed{	We proceed as in the proof before until \eqref{eq:ito_after_stopping_time} and then infer by \ref{a3le2} in
	\autoref{rem:weaker_assumption_on_monotonicity}, with $\beta_1+\beta_2=:\beta$, that}
	\begin{align*}
	&\int_t^{\tau_n}|Z_s|^2ds+\int_t^{\tau_n}\|U_s\|^2 ds\leq|Y_{\tau_n}|^2+2\int_t^{\tau_n}\left(\alpha(s)\rho(|Y_s|^p)|Y_s|^{2-p}+\mu(s)|Y_s|^2\right)ds\\
	&+\int_t^{\tau_n}\beta(s)|Y_s|(|Z_s|+\|U_s\|)ds+2\int_t^{\tau_n}|Y_s||\fzero(s)| ds\\
	&+2\left|\int_t^{\tau_n} \tempinnerproduct{Y_s}{Z_sdW_s}\right|+\left|\int_{{]t,\tau_n]}\times\Rzerod}\left(|Y_{s-}+U_s(x)|^2-|Y_{s-}|^2\right)\tilde{N}(ds,dx)\right|.
	\end{align*}
	Taking the power $\frac{p}{2}$, we find a constant $\czero > 0$ such that
	\begin{align*}
	&\left[\int_t^{\tau_n}|Z_s|^2ds\right]^\frac{p}{2}+\left[\int_t^{\tau_n}\|U_s\|^2 ds\right]^\frac{p}{2}\\
	&\leq \czero\Biggl(|Y_{\tau_n}|^p+\left[\sup_{s \in [t,T]}|Y_s|^{2-p} \int_t^{\tau_n}\alpha(s)\rho(|Y_s|^{p})ds+
	\sup_{s \in [t,T]}|Y_s|^2\int_t^{\tau_n}\mu(s)ds\right]^\frac{p}{2}\\
	&\quad+\left[\int_t^{\tau_n}\beta(s)|Y_s|(|Z_s|+\|U_s\|)ds\right]^\frac{p}{2}+\left[\int_t^{\tau_n}|Y_s||\fzero(s)|ds\right]^\frac{p}{2}\\
	&\quad+\left|\int_t^{\tau_n} \tempinnerproduct{Y_s}{Z_sdW_s}\right|^\frac{p}{2}+\left|\int_{{]t,\tau_n]}\times\Rzerod}\left(|Y_{s-}+U_s(x)|^2-|Y_{s-}|^2\right)\tilde{N}(ds,dx)\right|^\frac{p}{2}\Biggr).
	\end{align*}
	
	We estimate further with $\cone > 0$
	\begin{align*}
	&\left[\int_t^{\tau_n}|Z_s|^2ds\right]^\frac{p}{2}+\left[\int_t^{\tau_n}\|U_s\|^2 ds\right]^\frac{p}{2}\leq \cone\Biggl(|Y_{\tau_n}|^p+\sup_{s \in [t,T]}|Y_s|^{\frac{(2-p)p}{2}} \left[\int_t^{\tau_n}\alpha(s)\rho(|Y_s|^{p})ds\right]^{\frac{p}{2}}\\
	&+
	\sup_{s \in [t,T]}|Y_s|^{p}\left[\int_t^{\tau_n}\mu(s)ds\right]^\frac{p}{2} +\left[\int_t^{\tau_n}\beta(s)|Y_s|(|Z_s|+\|U_s\|)ds\right]^\frac{p}{2}+\left[\int_t^{\tau_n}|Y_s||\fzero(s)|ds\right]^\frac{p}{2}\\
	&+\left|\int_t^{\tau_n} \tempinnerproduct{Y_s}{Z_sdW_s}\right|^\frac{p}{2}+\left|\int_{{]t,\tau_n]}\times\Rzerod}\left(|Y_{s-}+U_s(x)|^2-|Y_{s-}|^2\right)\tilde{N}(ds,dx)\right|^\frac{p}{2}\Biggr).
	\end{align*}
	
	With Young's inequality for $(\frac{2}{p},\frac{2}{2-p})$ and a new constant $c_2 > 0$ we get,
	\begin{align*}
	&\left[\int_t^{\tau_n}|Z_s|^2ds\right]^\frac{p}{2}+\left[\int_t^{\tau_n}\|U_s\|^2 ds\right]^\frac{p}{2}\leq c_2\Biggl(|Y_{\tau_n}|^p+\sup_{s \in [t,T]}|Y_s|^{p} +   \int_t^{\tau_n}\alpha(s)\rho(|Y_s|^{p})ds\\
	&+
	 \left[\int_t^{\tau_n}\beta(s)|Y_s|(|Z_s|+\|U_s\|)ds\right]^\frac{p}{2}+\left[\int_t^{\tau_n}|Y_s||\fzero(s)|ds\right]^\frac{p}{2}\\
	&+\left|\int_t^{\tau_n} \tempinnerproduct{Y_s}{Z_sdW_s}\right|^\frac{p}{2}+\left|\int_{{]t,\tau_n]}\times\Rzerod}\left(|Y_{s-}+U_s(x)|^2-|Y_{s-}|^2\right)\tilde{N}(ds,dx)\right|^\frac{p}{2}\Biggr).
	\end{align*}
	
	From here on the proof can be concluded similar to case 1 of \autoref{YdeterminesZU}.
\end{proof}

From the proposition above, we now know how to bound $Z$ and $U$ in terms of $Y$ and $f_0$. For the core of the existence proof later we need to control the $Y$ part of the solution triplet by a bound depending only on $\xi$ and $f$, which we will show in the sequel.

\begin{proposition}\label{supprop}
	Let $p\geq 2$ and let $(Y,Z,U)$ be an $\lp$-solution to the BSDE $(\xi,f)$. If $\xi\in L^p$ and \ref{A1} and \ref{a3ge2} are satisfied,
	then there exists a  function $h:[0,\infty[\to [0,\infty[ $ with
	$h(x)\to 0$ as $x\to 0$ 
	such that
	\begin{align*}
	&\|Y\|^p_{\mathcal{S}^p} 
	\leq h\left(\E|\xi|^p+\E \Ifzero^p\right),
	\end{align*}	
	where $h$ depends on $p, T, \rho, \alpha, \beta, \mu$.
\end{proposition}\begin{proof}
	
{\bf Step 1:}\\	
	Let $\Psi(y) := |y|^p$ and $\eta$ = $(\eta_t)_{0 \le t \le T} \in L^\infty(\Omega; L^1([0,T]))$ be a progressively measurable, continuous process, which we will determine later. It\^o's formula (see also \cite[Proposition 2]{Kruse}) for $t \in [0,T]$ implies
	\begin{align}	\label{gammaeq1}
 	&\ietat|Y_t|^p + \int_t^T \ietas\Big[\eta(s)|Y_s|^p + \frac{1}{2} \trace(D^2 \Psi(Y_s)Z_sZ_s^*) \Big]  ds + P(t) \nonumber \\
 	&= \ietaT|\xi|^p + \int_t^T \ietas p \tempinnerproduct{Y_s |Y_s|^{p-2}}{\fsyzu} ds + M(t),
	\end{align}
	where $D^2 \Psi$ denotes the Hessian matrix of $\Psi$,
	\changed{\begin{align*}
	P(t) = & \int_t^T \int_{\Rzerod} \ietas\left[ |Y_{s-}+U(s,x)|^p - |Y_{s-}|^p - \tempinnerproduct{U(s,x)}{p Y_{s-}|Y_{s-}|^{p-2}}\right]\nu(dx)ds
	\end{align*}}
	and
	\begin{align*}
	M(t) = & - \int_t^T \ietas p \tempinnerproduct{Y_{s}|Y_{s}|^{p-2}}{Z_s dW_s}\nonumber \\
	&- \int_{]t,T]\times\Rzerod} \ietas\left[ |Y_{s-}+U(s,x)|^p - |Y_{s-}|^p\right]\tilde{N}(ds,dx).
	\end{align*}
	
	By the argument in \changed{\cite[Proposition 2]{Kruse}} we can use the estimates
	$\trace(D^2 \Psi(y)zz^*) \ge p|y|^{p-2}|z|^2$ and
	$$
	P(t) \ge p(1-p)3^{1-p}\int_t^T \ietas |Y_{s-}|^{p-2} \|U_s\|^2 ds \,,
	$$
	
	leading to
	\begin{align*}
		&\ietat|Y_t|^p + \int_t^T \ietas\Big[\eta(s)|Y_s|^p + \frac{1}{2} p(p-1)|Y_s|^{p-2} |Z_s|^2 \Big]  ds \nonumber \\
		&\quad+p(1-p)3^{1-p}\int_t^T \ietas |Y_{s}|^{p-2} \|U_s\|^2 ds  \nonumber  \\
		&\le\ietaT|\xi|^p + \int_t^T \ietas p \tempinnerproduct{Y_s |Y_s|^{p-2}}{\fsyzu }ds +  M(t)		 \,.
	\end{align*}
		
	\changed{Using $c_z = \frac{1}{2}p(1-p)$ and $c_u = p(1-p)3^{1-p}$, \ref{a3ge2}, \autoref{remA3}\ref{remA3:A3bounds}\ref{remA3ge2}, Young's inequality for arbitrary $R_z, R_u > 0$ with the conjugate couple $(2,2)$, and Young's inequality once more for the expression ${|Y_s|^{p-1}|\fzero| = |Y_s|^{p-1}K_{|f_0|}^{(p-1)/p} |\fzero|^{1/p}\Ifzero^{(p-1)/p}}$, (see \eqref{eq:i_f_and_k_f}), and the couple $(\frac{p}{p-1},p)$, we find}
	\begin{align}\label{gammaeq2}
		&e^{\int_0^t\eta(s)ds}{|Y_t|}^p+\int_t^T e^{\int_0^s\eta(\tau)d\tau}\left(\eta(s){|Y_s|}^p +c_z|Y_s|^{p-2}|Z_s|^2 +c_u{|Y_{s}|^{p-2}}\|U_s\|^2\right)ds\nonumber\\
		& \leq  e^{\int_0^T\eta(s)ds}|\xi|^p+\int_t^T\! e^{\int_0^s\eta(\tau)d\tau}p\biggl(
		\alpha(s)\rho(|Y_s|^p) 
		+\left(\alpha(s)\rho(1) + \mu(s)+\frac{R_z+R_u}{2}\beta(s)^2\right)|Y_s|^p \biggr)ds \nonumber \\
		&\quad+\int_t^T e^{\int_0^s\eta(\tau)d\tau}p|Y_s|^{p-2}\left(\frac{|Z_s|^2}{2R_z}+\frac{\|U_s\|^2}{2R_u}\right) ds+
		\int_t^T e^{\int_0^s\eta(\tau)d\tau}  		
		(p-1)|Y_s|^p\Kfzero(s)ds
		  \nonumber\\
		&\quad +\int_t^T e^{\int_0^s\eta(\tau)d\tau} |\fzero(s)| \Ifzero^{p-1}ds+M(t).
	\end{align}
	
	We set
	$R_z = p/c_z$, $R_u = p/c_u$ and $\eta = p(\alpha\rho(1)+ \mu+\beta^2(R_z+R_u)/2) + (p-1)\Kfzero$ leading to	
	\begin{align*}
		&e^{\int_0^t\eta(s)ds}{|Y_t|}^p +
		\int_t^T e^{\int_0^s\eta(\tau)d\tau}\left(\frac{c_z}{2}|Y_s|^{p-2}|Z_s|^2 +\frac{c_u}{2}{|Y_{s}|^{p-2}}\|U_s\|^2\right)ds
		\nonumber\\
		& \leq  e^{\int_0^T\eta(s)ds}|\xi|^p+\int_t^T e^{\int_0^s\eta(\tau)d\tau}p
		\alpha(s)\rho(|Y_s|^p)
		 ds +\int_t^T e^{\int_0^s\eta(\tau)d\tau} |\fzero(s)| \Ifzero^{p-1}ds+M(t).	
	\end{align*}

	Now, we omit $e^{\int_0^t\eta(s)ds}{|Y_t|}^p$ and take expectations,	
	\begin{align*}
	&\E \int_t^T e^{\int_0^s\eta(\tau)d\tau}\left(\frac{c_z}{2}|Y_s|^{p-2}|Z_s|^2 +\frac{c_u}{2}{|Y_{s}|^{p-2}}\|U_s\|^2\right)ds
	\nonumber\\
	& \leq \E e^{\int_0^T\eta(s)ds}|\xi|^p
	+\E\int_t^T e^{\int_0^s\eta(\tau)d\tau}p
	\alpha(s)\rho(|Y_s|^p)
	ds + \E \Ifzero^{p-1} \int_t^T e^{\int_0^s\eta(\tau)d\tau} |\fzero(s)| ds.	
	\end{align*}

Hence, we find a constant $\czero > 0$,
to end the step with
\begin{align}\label{eq:apriorige2ZUbound}
&\E \int_t^T \left(|Y_s|^{p-2}|Z_s|^2 +{|Y_{s}|^{p-2}}\|U_s\|^2\right)ds\leq \czero \left( \E|\xi|^p
+\E\int_t^T
\alpha(s)\rho(|Y_s|^p)
ds  + \E \Ifzero^{p}.	\right).
\end{align}
	{\bf Step 2:}\\	
	\changed{
	We take the same route as in the previous step until \eqref{gammaeq2}, with one difference: we keep the $P(t)$ term, to get
	\begin{align}\label{eq:apriorige2step2}
	&e^{\int_0^t\eta(s)ds}{|Y_t|}^p+\int_t^T e^{\int_0^s\eta(\tau)d\tau}(\eta(s){|Y_s|}^p +c_z|Y_s|^{p-2}|Z_s|^2 )ds \nonumber \\
	& \leq  e^{\int_0^T\eta(s)ds}|\xi|^p+\int_t^T\! e^{\int_0^s\eta(\tau)d\tau}p\biggl(
	\alpha(s)\rho(|Y_s|^p) 
	+\left(\alpha(s)\rho(1) + \mu(s)+\frac{R_z+R_u}{2}\beta(s)^2\right)|Y_s|^p \biggr)ds \nonumber \\
	&\quad+\int_t^T e^{\int_0^s\eta(\tau)d\tau}p|Y_s|^{p-2}\left(\frac{|Z_s|^2}{2R_z}+\frac{\|U_s\|^2}{2R_u}\right) ds+
	\int_t^T e^{\int_0^s\eta(\tau)d\tau}  		
	(p-1)|Y_s|^p\Kfzero(s)ds
	\nonumber\\
	&\quad +\int_t^T e^{\int_0^s\eta(\tau)d\tau} |\fzero(s)| \Ifzero^{p-1}ds+M(t) - P(t).
	\end{align}
	 
	Now we set
	$R_z = p/(2c_z)$, $R_u = 1/2$ and $\eta = p(\alpha\rho(1)+ \mu+\beta^2(R_z+R_u)/2) + (p-1)\Kfzero$. By the choice of a suitable constant $\cone > 0$	

	\begin{align*}
	&e^{\int_0^t\eta(s)ds}{|Y_t|}^p \leq  e^{\int_0^T\eta(s)ds}|\xi|^p+\int_t^T e^{\int_0^s\eta(\tau)d\tau}p\left(
	\alpha(s)\rho(|Y_s|^p) + |Y_s|^{p-2}\|U\|^2 \right)
	ds + \cone  \Ifzero^{p} + M(t) - P(t).	
	\end{align*}
	
	We can rewrite
	\begin{align*}
    M(t) - P(t) = &-\int_{t}^{T} \ietas p \tempinnerproduct{Y_{s}|Y_{s}|^{p-2}}{Z_s dW_s} \\
    & -\int_{]t,T]\times\Rzerod} \ietas\left[ |Y_{s-}+U(s,x)|^p - |Y_{s-}|^p - \tempinnerproduct{U(s,x)}{p Y_{s-}|Y_{s-}|^{p-2}}\right]N(ds,dx)\\
	& - \int_{]t,T]\times\Rzerod} \ietas \tempinnerproduct{U(s,x)}{p Y_{s-}|Y_{s-}|^{p-2}}\tilde{N}(ds,dx) .
	\end{align*}
	By Taylor expansion of $|\cdot|^p$ (see \cite[Proposition 2]{Kruse})
	$$ |Y_{s-}+U(s,x)|^p - |Y_{s-}|^p - \tempinnerproduct{U(s,x)}{p Y_{s-}|Y_{s-}|^{p-2}} \ge 0 .
	$$ With the minus in front, we can omit the integral with respect to $N(ds,dx)$ in \eqref{eq:apriorige2step2}, take suprema and end up with
	
	\begin{align}\label{eq:apriorisupMtfilling}
		\sup_{s\in{[t,T]}}e^{\int_0^s\eta(\tau)d\tau}{|Y_s|}^p &\leq  e^{\int_0^T\eta(s)ds}|\xi|^p+\int_t^T e^{\int_0^s\eta(\tau)d\tau}p\left(
		\alpha(s)\rho(|Y_s|^p) + |Y_s|^{p-2}\|U\|^2 \right)ds + \cone \Ifzero^{p} 
		\nonumber\\
		  & \quad + \sup_{r\in{[t,T]}} 
		  \left|\int_{t}^{r} \ietas p \tempinnerproduct{Y_{s}|Y_{s}|^{p-2}}{Z_s dW_s} \right| \nonumber\\
		  & \quad + \sup_{r\in{[t,T]}} \left|
		   \int_{]t,r]\times\Rzerod} \ietas \tempinnerproduct{U(s,x)}{p Y_{s-}|Y_{s-}|^{p-2}}\tilde{N}(ds,dx)\right|.
	\end{align}
	We proceed by estimating the expectation of these two suprema in the next step}.\bigskip

{\bf Step 3:}\\
	\changed{For the first supremum, we apply the Burkholder-Davis-Gundy inequality (\cite[Theorem 10.36]{HeWangYan}) giving $\ctwo > 0$  and the first line of the following inequality. Then, we pull out $\sup_{s\in{[t,T]}} |Y_s|^{\frac{p}{2}}$ from the $ds$-integral (and the squareroot) and finally use Young's inequality for arbitrary $R > 0$, to estimate
		\begin{align}\label{eq:first_supremum_estimate}
		\E\sup_{r\in [t,T]}\left|\int_{t}^{r} \ietas p \tempinnerproduct{Y_{s}|Y_{s}|^{p-2}}{Z_s dW_s} \right| & \le \ctwo \E \left( \int_t^T (\ietas p |Y_{s}|^{p-1} |Z_s|)^2 ds\right)^{1/2}\nonumber \\
		&\le \ctwo p \E \sup_{s\in{[t,T]}} |Y_s|^{\frac{p}{2}} \left(\int_t^T e^{2\int_0^s\eta(\tau)d\tau} |Y_s|^{p-2} |Z_s|^2 ds\right)^{1/2}\nonumber\\
		&\le \cthree \E \left(\frac{1}{R} \sup_{s\in{[t,T]}} |Y_s|^{p} + R \int_t^T |Y_s|^{p-2}|Z_s|^2 ds \right) ,
		\end{align}
		for another constant $\cthree > 0$.
		
		For the second suprema, we can use \cite[Theorem 3.2]{MarinRoeck} to get $\cfour > 0$ such that
		\begin{align*}
		&\E \sup_{r \in [t,T]}  \left|
		\int_{]t,r]\times\Rzerod} \ietas \tempinnerproduct{U(s,x)}{p Y_{s-}|Y_{s-}|^{p-2}}\tilde{N}(ds,dx)\right| \\
		&\le \cfour \E \left(\int_t^T \int_{\Rzerod} e^{2\int_0^s\eta(\tau)d\tau} \left(|U(s,x)|p |Y_{s}||Y_{s}|^{p-2}\right)^2  \nu(dx)ds\right)^\frac{1}{2} \\ 
		&\le \cfive \E \left(\frac{1}{R} \sup_{s\in{[t,T]}} |Y_s|^{p} + R \int_t^T |Y_s|^{p-2}\|U_s\|^2 ds \right) \,.
		\end{align*}			
		In the last step, we used Young's inequality as above in \eqref{eq:first_supremum_estimate} for some arbitrary $R > 0$ to get the constant $\cfive > 0$.}
	
	{\bf Step 4:}\\
	With the last step's results we continue from \eqref{eq:apriorisupMtfilling} to get a constant $D > 0$ satisfying	
	\begin{align}\label{eq:apriorisupMtdone}
	&\E \sup_{s\in{[t,T]}}e^{\int_0^s\eta(\tau)d\tau}{|Y_s|}^p \leq D \E \Biggl( |\xi|^p+\int_t^T
	\alpha(s)\rho(|Y_s|^p)ds +  \Ifzero^{p} \nonumber\\
	& \quad\quad  +\frac{1}{R} \sup_{s\in{[t,T]}} |Y_s|^{p} + R \int_t^T |Y_s|^{p-2}|Z_s|^2 ds + R\int_t^T |Y_s|^{p-2}\|U_s\|^2 ds \Biggr) .
	\end{align}
	
	We apply inequality \eqref{eq:apriorige2ZUbound} yielding
	\begin{align*}
	&\E \sup_{s\in{[t,T]}}e^{\int_0^s\eta(\tau)d\tau}{|Y_s|}^p \leq D(1+Rk) \E \Biggl( |\xi|^p+\int_t^T 
	\alpha(s)\rho(|Y_s|^p)ds +  \Ifzero^{p}\Biggr) 
	+   \frac{D}{R} \E \sup_{s\in{[t,T]}} |Y_s|^{p} .
	\end{align*}
	
	We choose $R = 2D$, which implies that there is $D_1>0$ such that
	\begin{align*}
	&\E \sup_{s\in{[t,T]}}{|Y_s|}^p\leq \E \sup_{s\in{[t,T]}}e^{\int_0^s\eta(\tau)d\tau}{|Y_s|}^p \leq D_1 \Biggl( \E |\xi|^p+  \E \Ifzero^{p}+\int_t^T 
	\alpha(s)\rho(\E\sup_{r\in{[s,T]}}|Y_r|^p)ds \Biggr),
	\end{align*}
	 where we also used the concavity of $\rho$. Now, the Bihari-LaSalle inequality (see  \autoref{bihari-prop} in the Appendix) finishes the proof.
	\end{proof}

\begin{proposition}\label{supproppleq2}
	Let $1<p<2$ and let $(Y,Z,U)$ be an $\lp$-solution to the BSDE $(\xi,f)$. If $\xi\in L^p$ and \ref{A1} and \ref{a3le2} are satisfied,
	then there exists a  function $h:[0,\infty[\to [0,\infty[ $ with
	$h(x)\to 0$ as $x\to 0$ 
	such that
	\begin{align*}
	&\|Y\|^p_{\mathcal{S}^p} 
	+\left\|Z\right\|_{L^p(W) }^p + \left\|U\right\|_{L^p(\tilde N) }^p 
	\leq h\left(\E|\xi|^p+\E \Ifzero^p\right),
	\end{align*}	
	where $h$ depends on $p, T, \rho, \alpha, \beta_1, \beta_2, \mu$.
\end{proposition}

\begin{proof}
{\bf Step 1:}\\
	We begin this proof similarly to the case $p\geq 2$: Let $\eta$ be a progressively measurable process in $L^\infty(\Omega; L^1([0,T]))$, which we will determine later. As carried out in detail in \cite[Proposition 3]{Kruse}, It\^o's formula, applied to the smooth function $u_\varepsilon: x\mapsto (|x|^2+\varepsilon)^\frac{p}{2}$ and taking the limit $\varepsilon\to 0$ implies that for $\czero = \frac{p(p-1)}{2}$ and $t \in [0,T]$,
	\begin{align*}	
 	&\ietat|Y_t|^p + \int_t^T \ietas\Big[\eta(s)|Y_s|^p + \czero|Y_s|^{p-2} |Z_s|^2 \chi_{\{ Y_s \neq 0\}} \Big]  ds + P(t) \nonumber \\
 	&\le M(t) + \ietaT|\xi|^p + \int_t^T \ietas p \tempinnerproduct{Y_s |Y_s|^{p-2}}{\fsyzu} ds,
	\end{align*}
	where 
	\begin{align*}
	P(t) = & \int_t^T \int_{\Rzerod} \ietas\left[ |Y_{s-}+U_s(x)|^p - |Y_{s-}|^p - \tempinnerproduct{U_s(x)}{pY_{s-}|Y_{s-}|^{p-2}}\right]\nu(dx)ds
	\end{align*}
	and
	\begin{align*}
	M(t) = & - \int_t^T \ietas p \tempinnerproduct{Y_{s}|Y_{s}|^{p-2}}{Z_s dW_s}\nonumber \\
	&- \int_{]t,T]\times\Rzerod} \ietas\left[ |Y_{s-}+U_s(x)|^p - |Y_{s-}|^p\right]\tilde{N}(ds,dx).
	\end{align*}
	The terms slightly differ from \cite[Proposition 3]{Kruse}. The alternative expressions are due to the relation $d\tilde{N}(ds,dx)=dN(ds,dx)-\nu(dx)dt$, used to split up the integrals w.r.t. those random measures accordingly in the limit procedure. We may do this as all relevant integrands appearing in the It\^o formula for $u_\varepsilon$ and in the limit expression yield $\mathbb{P}$-a.s. finite integrals. Their finiteness results from the convexity of $u_\varepsilon$, the boundedness of its second derivative for all $\varepsilon>0$ and from the fact that $|Y_{s-}|\vee|Y_{s-}+U_s(x)|\leq 4\sup_{t\in [0,T]}|Y_t|$.

	By the argument in \cite[Proposition 3]{Kruse} we can use the estimate
	\begin{align*}
	P(t) \ge \czero\!\!\int_t^T\!\! \ietas\!\! \int_{\Rzerod}\!\!\!(|Y_{s-}|\vee|Y_{s-}+U_s(x)|)^{p-2} |U_s(x)|^2\chi_{(|Y_{s-}|\vee|Y_{s-}+U_s(x)|) \neq 0\}}\nu(dx) ds \,
	\end{align*}
	leading to
	\begin{align*}
		&\ietat|Y_t|^p + \int_t^T \ietas\Big[\eta(s)|Y_s|^p + \czero|Y_s|^{p-2} |Z_s|^2\chi_{\{ Y_s \neq 0\}} \Big]  ds \nonumber \\
		&+\czero\int_t^T\! \ietas \int_{\Rzerod}\!(|Y_{s-}|\vee|Y_{s-}+U_s(x)|)^{p-2} |U_s(x)|^2\chi_{(|Y_{s-}|\vee|Y_{s-}+U_s(x)|) \neq 0\}}\nu(dx)ds  \nonumber  \\
		&\le M(t) + \ietaT|\xi|^p + \int_t^T \ietas p \tempinnerproduct{Y_s |Y_s|^{p-2}}{\fsyzu} ds		 \,.
	\end{align*}

Using \ref{a3le2} 
and Young's inequality, we obtain for an arbitrary $R_z > 0$,
	\begin{align}\label{gammaeq2pleq2}
		&e^{\int_0^t\eta(s)ds}{|Y_t|}^p+\int_t^T e^{\int_0^s\eta(\tau)d\tau}\left(\eta(s){|Y_s|}^p +\czero|Y_s|^{p-2}|Z_s|^2\chi_{\{ Y_s \neq 0\}}\right)ds\nonumber \\
		&+\czero\int_t^T \!\ietas\int_{\Rzerod}\!(|Y_{s-}|\vee|Y_{s-}+U_s(x)|)^{p-2} |U_s(x)|^2\chi_{(|Y_{s-}|\vee|Y_{s-}+U_s(x)|) \neq 0\}}\nu(dx)ds\nonumber\\
		& \leq  e^{\int_0^T\eta(s)ds}|\xi|^p+\int_t^T \!e^{\int_0^s\eta(\tau)d\tau}p\biggl(
		\alpha(s)\rho(|Y_s|^p) 
		+\left(\mu(s)\!+\!\frac{R_z \beta_1(s)^2}{2}+(p-1)\Kfzero(s)\!\right)|Y_s|^p\! \biggr)ds \nonumber \\
		&\quad + \int_t^T e^{\int_0^s\eta(\tau)d\tau}\frac{p}{2 R_z}|Y_s|^{p-2}|Z_s|\chi_{\{ Y_s \neq 0\}}ds+\int_t^T e^{\int_0^s\eta(\tau)d\tau}p\beta_2(s)|Y_s|^{p-1}\|U_s\| ds\nonumber\\
		&\quad +
		\int_t^T e^{\int_0^s\eta(\tau)d\tau} |\fzero(s)| \Ifzero^{p-1}ds+M(t).
	\end{align}
We choose $R_z=\frac{p}{\czero}$ and $\eta=p\left(\mu+\frac{R_z \beta_1^2}{2}+(p-1)\Kfzero\right)$, take expectations and omit the first term to arrive at
	\begin{align*}
		&\frac{\czero}{2}\E\int_t^T e^{\int_0^s\eta(\tau)d\tau}|Y_s|^{p-2}|Z_s|^2\chi_{\{ Y_s \neq 0\}}ds\nonumber \\
		&+\czero\E\int_t^T \!\!\ietas\!\int_{\Rzerod}\!\!\!(|Y_{s-}|\vee|Y_{s-}+U_s(x)|)^{p-2} |U_s(x)|^2 \chi_{(|Y_{s-}|\vee|Y_{s-}+U_s(x)|) \neq 0\}}\nu(dx)ds\nonumber\\
		& \leq  \E e^{\int_0^T\eta(s)ds}|\xi|^p+\E\int_t^T e^{\int_0^s\eta(\tau)d\tau}p
		\alpha(s)\rho(|Y_s|^p) 
		ds \nonumber \\
		&\quad+\E\int_t^T e^{\int_0^s\eta(\tau)d\tau}p\beta_2(s)|Y_s|^{p-1}\|U_s\|ds+
		\E\int_t^T e^{\int_0^s\eta(\tau)d\tau} |\fzero(s)| \Ifzero^{p-1}ds,
	\end{align*}	
	yielding a constant $C>0$ such that
	\begin{align}\label{gammaeq3pleq2}
	&\E\int_t^T |Y_s|^{p-2}|Z_s|^2\chi_{\{ Y_s \neq 0\}}ds\nonumber \\
		&\quad +\E\int_t^T \int_{\Rzerod}(|Y_{s-}|\vee|Y_{s-}+U_s(x)|)^{p-2} |U_s(x)|^2 \chi_{(|Y_{s-}|\vee|Y_{s-}+U_s(x)|) \neq 0\}}\nu(dx)ds\nonumber\\
		& \leq C\biggl(\E |\xi|^p+\E\int_t^T 
		\alpha(s)\rho(|Y_s|^p) 
		ds +\E\int_t^T \beta_2(s)|Y_s|^{p-1}\|U_s\|ds+
		\E \Ifzero^p\biggr).
	\end{align}
	
{\bf Step 2:}\\
In this step we leave the argumentation lines of Kruse and Popier \cite{Kruse,Kruse17} and \changed{Briand et al \cite{briand2003lp}, estimating several terms differently and using the integrability assumptions on $\beta_2$.} 
We start from estimating the suprema of the stochastic integrals appearing in \eqref{gammaeq2pleq2} by similar means as in step 3 of the proof of \autoref{supprop},  \eqref{eq:apriorisupMtfilling} - \eqref{eq:apriorisupMtdone} which yields constants $c,\cone > 0$, such that for an arbitrary $R>0$ we get \changed{
\begin{align}\label{eq:gammaeq2pleq2M}
&\sup_{s\in [t,T]}|M(s)| \leq c\biggl(\frac{1}{R}\E\sup_{s\in{[t,T]}}|Y_s|^p+R\E\int_t^T |Y_s|^{p-2}|Z_s|^2\chi_{\{Y_s\neq 0\}}ds\nonumber\\
&\quad +\E\biggl[\int_t^T\int_{\Rzerod}\left(\left(|Y_{s-}+U_s(x)|\vee|Y_{s-}|\right)^{p-1}|U_s(x)|\right)^2\chi_{(|Y_{s-}|\vee|Y_{s-}+U_s(x)|) \neq 0\}}\nu(dx)ds\biggr]^\frac{1}{2}\biggr)\nonumber\\
&\ \leq c_1\biggl(\frac{1}{R}\E\sup_{s\in{[t,T]}}|Y_s|^p+R\E\int_t^T |Y_s|^{p-2}|Z_s|^2\chi_{\{Y_s\neq 0\}}ds\nonumber\\
&\quad +R\E\int_t^T\int_{\Rzerod}\left(|Y_{s-}+U_s(x)|\vee|Y_{s-}|\right)^{p-2}|U_s(x)|^2\chi_{(|Y_{s-}|\vee|Y_{s-}+U_s(x)|) \neq 0\}}\nu(dx)ds\biggr),
\end{align}}
where again we used Young's inequality as well as that $|Y_{s-}+U_s(x)|\vee|Y_{s-}|\leq 4\sup_{s\in{[t,T]}}|Y(s)|$, $\mathbb{P}\otimes\lambda\otimes\nu$-a.e.

Taking suprema in \eqref{gammaeq2pleq2} with the same choices  $R_z=\frac{p}{c_0}$ and $\eta=p\left(\mu+\frac{R_z \beta_1^2}{2}+(p-1)\Kfzero\right)$ (to cancel out appropriate terms) yields  
\begin{align*}
		&\sup_{s\in{[t,T]}}e^{\int_0^s\eta(\tau)d\tau}{|Y_s|}^p\nonumber \\
		&+\czero\int_t^T \!\ietas\int_{\Rzerod}\!(|Y_{s-}|\vee|Y_{s-}+U_s(x)|)^{p-2} |U_s(x)|^2\chi_{(|Y_{s-}|\vee|Y_{s-}+U_s(x)|) \neq 0\}}\nu(dx)ds\nonumber\\
		& \leq  e^{\int_0^T\eta(s)ds}|\xi|^p+\int_t^T \!e^{\int_0^s\eta(\tau)d\tau}p
		\alpha(s)\rho(|Y_s|^p)ds \nonumber \\
		&\quad +\int_t^T e^{\int_0^s\eta(\tau)d\tau}p\beta_2(s)|Y_s|^{p-1}\|U_s\| ds +
		\int_t^T e^{\int_0^s\eta(\tau)d\tau} |\fzero(s)| \Ifzero^{p-1}ds+\sup_{s\in [t,T]}|M(s)|,
	\end{align*}
where we omitted the remaining positive integral terms $\frac{\czero}{2}\E\int_t^T e^{\int_0^s\eta(\tau)d\tau}|Y_s|^{p-2}|Z_s|^2\chi_{\{ Y_s \neq 0\}}ds$ and 
$$\czero\int_t^T \!\ietas\int_{\Rzerod}\!(|Y_{s-}|\vee|Y_{s-}+U_s(x)|)^{p-2} |U_s(x)|^2\chi_{(|Y_{s-}|\vee|Y_{s-}+U_s(x)|) \neq 0\}}\nu(dx)ds$$
that appear on the left hand side after cancelling the integrals  involving $|Y_s|^{p-2}|Z_s|^2$ and $|Y(s)|^p$ (not the one with $\rho$) on the right hand side of \eqref{gammaeq2pleq2} after substituting $R_z$ and $\eta$.

 Using the above estimate \eqref{eq:gammaeq2pleq2M} for $\sup_{s\in [t,T]}|M(s)|$, we take expectations and come to
\begin{align*}
		&\E \sup_{s\in{[t,T]}}e^{\int_0^s\eta(\tau)d\tau}{|Y_s|}^p \nonumber\\
		& \leq  \E e^{\int_0^T\eta(s)ds}|\xi|^p+\E\int_t^T e^{\int_0^s\eta(\tau)d\tau}p
		\alpha(s)\rho(|Y_s|^p)ds \nonumber \\
		&\quad+\E\int_t^T e^{\int_0^s\eta(\tau)d\tau}p|Y_s|^{p-1}\beta_2(s)\|U_s\|ds+
		\E\int_t^T e^{\int_0^s\eta(\tau)d\tau} |\fzero(s)| \Ifzero^{p-1}ds\\
		&\quad+\frac{\cone}{R}\sup_{s\in{[t,T]}}{|Y_s|}^p+\cone R\biggl(\E\int_t^T |Y_s|^{p-2}|Z_s|^2\chi_{\{ Y_s \neq 0\}}ds\nonumber \\
		&\quad +\E\int_t^T \int_{\Rzerod}(|Y_{s-}|\vee|Y_{s-}+U_s(x)|)^{p-2} |U_s(x)|^2 \chi_{(|Y_{s-}|\vee|Y_{s-}+U_s(x)|) \neq 0\}}\nu(dx)ds\biggr).
	\end{align*}	
Now inequality \eqref{gammaeq3pleq2} can be plugged in for the last parentheses to estimate, for another constant $D>0$,
\begin{align}\label{gammaeq4pleq2}
		&\E \sup_{s\in{[t,T]}}{|Y_s|}^p \leq  D\biggl((1+R)\E |\xi|^p+(1+R)\E\int_t^T 
		\alpha(s)\rho(|Y_s|^p)ds \nonumber \\
		&\quad\quad+(1+R)\E\int_t^T\beta_2(s)|Y_s|^{p-1}\|U_s\|ds+
		(1+R)\E\int_t^T  |\fzero(s)| \Ifzero^{p-1}ds+\frac{1}{R}\E\sup_{s\in{[t,T]}}{|Y_s|}^p\biggr).
	\end{align}	
We focus on the term $\E\int_t^T\beta_2(s)|Y_s|^{p-1}\|U_s\|ds$, which we estimate by 
\begin{align*}
&\E\sup_{s\in{[t,T]}}|Y_s|^{p-1}\int_t^T\beta_2(s)\|U_s\| ds\\
&\leq \frac{\ctwo}{R_1}\E\sup_{s\in{[t,T]}}|Y_s|^p+ \ctwo R_1^{p-1}\E\left[\int_t^T\beta_2(s)\|U_s\| ds\right]^p,
\end{align*}
for $R_1 > 0$ and a constant $\ctwo >0$ coming from Young's inequality for the couple $(p,\frac{p}{p-1})$.
\changed{By the Cauchy-Schwarz inequality we get}  
\begin{align}\label{eq:gamma4hoelderpleq2}
&\E\sup_{s\in{[t,T]}}|Y_s|^{p-1}\int_t^T\beta_2(s)\|U_s\|ds\nonumber\\
&\leq \frac{\ctwo}{R_1}\E\sup_{s\in{[t,T]}}|Y_s|^p+2 \ctwo R_1^{p-1}\E\left[\left(\int_t^T\beta_2(s)^2ds\right)^\frac{p}{2}\left(\int_t^T\|U_s\|^2 ds\right)^\frac{p}{2}\right].
\end{align}
\changed{Now we use the additional integrability of $\beta_2$ with a power $q>2$. For the case $\beta_2 \in L^2([0,T])$ see \autoref{rem:beta_2_deterministic} after the proof. Here in the sequel we treat a non-deterministic $\beta_2$ where higher integrability is needed. It has the only purpose to obtain a factor containing $T-t$ by Hölder's inequality. Indeed, we infer from \eqref{eq:gamma4hoelderpleq2}}
\begin{align*}
&\E\sup_{s\in{[t,T]}}|Y_s|^{p-1}\int_t^T\beta_2(s)\|U_s\|ds\\
&\leq \frac{\ctwo}{R_1}\E\sup_{s\in{[t,T]}}|Y_s|^p+2 \ctwo R_1^{p-1}\E\left[(T-t)^{\frac{pq}{q-2}}\left(\int_t^T\beta_2(s)^qds\right)^\frac{p}{q}\left(\int_t^T\|U_s\|^2 ds\right)^\frac{p}{2}\right].
\end{align*}
Now, by the boundedness of $\int_0^T\beta_2(s)^qds$, and applying  \autoref{YdeterminesZUple2}, we get a constant $D_1>0$ such that
\changed{\begin{align}\label{eq:gammaeq5pleq2}
&\E\sup_{s\in{[t,T]}}|Y_s|^{p-1}\int_t^T\beta_2(s)\|U_s\| ds\\
&\leq \frac{\ctwo}{R_1}\E\sup_{s\in{[t,T]}}|Y_s|^p+D_1 R_1^{p-1}(T-t)^{\frac{pq}{q-2}}\left(\E\sup_{s\in{[t,T]}}|Y_s|^{p}+\E\int_t^T\alpha(s)\rho(|Y(s)|^p)ds+\E \Ifzero^p\right).\nonumber
\end{align}}
\changed{Inserting \eqref{eq:gammaeq5pleq2} into inequality \eqref{gammaeq4pleq2}, we get for a new constant $\tilde{D} > 0$
\begin{align*}
		\E \sup_{s\in{[t,T]}}{|Y_s|}^p &\leq  \tilde{D} \Biggl((1+R) \E |\xi|^p+(1+R)\E\int_t^T 
		\alpha(s)\rho(|Y_s|^p)ds  + (1+R)\E \Ifzero^p \nonumber\\
		&\quad+\left(\frac{1}{R}+\frac{(1+R)\ctwo}{R_1}+(1+R)D_1R_1^{p-1}(T-t)^\frac{pq}{q-2}\right)\E\sup_{s\in{[t,T]}}{|Y_s|}^p\Biggr).
	\end{align*}
	Now, choose $R$ such that $\frac{\tilde{D}}{R}<\frac{1}{2}$, afterwards choose $R_1$ such that $\frac{\tilde{D}(1+R)\ctwo}{R_1}<\frac{1}{4}$. Now, our goal for the next step is to divide $[0,T]$ into small parts in order to make the third term containing $(T-t)$ small too.}
	\bigskip
	
{\bf Step 3:}\\
From here on, let the time interval $[0,T]$ be partitioned into  $0=t_0<t_1<\dotsc<t_n=T$, such that for all $1\leq i\leq n$, \changed{$\tilde{D}(1+R)D_1R_1^{p-1}(t_i-t_{i-1})^\frac{pq}{q-2}<\frac{1}{8}$}. 
Thus, on the interval $[t_{n-1},T]$, we come to a constant $D_2>0$ such that
\begin{align*}
		\E \sup_{s\in{[t,T]}}{|Y_s|}^p \leq  D_2\biggl(\E |\xi|^p+\int_t^T 
		\alpha(s)\rho(\E|Y_s|^p)ds +\E \Ifzero^p\biggr).
\end{align*}
Now, the Bihari-LaSalle inequality (\autoref{bihari-prop}) shows that there is a function $h_n$ such that
\begin{align*}
\E \sup_{s\in{[t_{n-1},T]}}{|Y_s|}^p\leq h_n(\E |\xi|^p+\E \Ifzero^p).
\end{align*}
Performing the same steps as above for the interval $[t_{n-2},t_{n-1}]$, we find a function $h_{n-1}$ such that 
\begin{align*}\E \sup_{s\in{[t_{n-2},t_{n-1}]}}{|Y_s|}^p\leq h_{n-1}(\E |Y_{t_{n-1}}|^p+\E \Ifzero^p)\leq h_{n-1}\bigl(h_n(\E |\xi|^p+\E \Ifzero^p)+\E \Ifzero^p\bigr).
\end{align*}
Iterating the procedure backwards in time, we end up with functions $h_1,\dotsc,h_n$, accumulating to a function $\tilde{h}$, such that
\begin{align*}
\E \sup_{s\in{[0,T]}}{|Y_s|}^p\leq \tilde{h}(\E |\xi|^p+\E \Ifzero^p).
\end{align*}
The bound for $\|Z\|^p_{L^p(W)}+ \|U\|^p_{L^p(\tilde{N})}$ then follows from \autoref{YdeterminesZUple2}, concluding the proof.

\end{proof}
\begin{remark}\label{rem:beta_2_deterministic}
	In step 3, if $\beta_2$ is deterministic, we could impose the weaker condition, namely $\beta_2$ being only square-integrable (instead of in $L^q$, for some $q > 2$). Then we do not need to apply Hölder's inequality to \eqref{eq:gamma4hoelderpleq2} in order to choose the division of $[0,T]$ such that $D_1(t_i-t_{i-1})^\frac{pq}{q-2}$ is small. Instead we choose the partition such that the $\int_{t_{i-1}}^{t_i}\beta_2(s)^2ds$ become sufficiently small.
	\end{remark}

With the technique from the two a priori estimates above in hand, we can now prove another key part for the existence proof: boundedness stability of the $Y$ process, meaning that the solution process $Y$ stays bounded, when the data ($\xi$, $f$) has boundedness properties:

\begin{proposition}\label{prop:stability}
	Let $p>1 $ and $(Y,Z,U)$ be an $\lp$-solution to the BSDE $(\xi,f)$.
	
	If $\xi,\Ifzero \in L^\infty$ and
	\begin{enumerate}[label=({\roman*})]
		\item
		for $p\ge 2$, \ref{A1} and \ref{a3ge2} hold,
		\item
		or for $1 < p < 2$, \ref{A1} and \ref{a3le2} hold, 				
	\end{enumerate}
	
	then there exists a constant $C > 0$
	such that for all $t\in {[0,T]}$, $$|Y_t|^p\leq C ,\quad\mathbb{P}\text{-a.s.}$$
		
	Here $C$ depends on $\|\xi\|_\infty,\|\Ifzero\|_\infty,\mu,p,T,\rho,\alpha,\beta$ for $p \ge 2$. If $p<2$, C depends on the same variables but $\beta_1,\beta_2,q$ instead of $\beta$.
\end{proposition}

\begin{proof}	
	\changed{We copy the proofs of \autoref{supprop} and \autoref{supproppleq2} for the cases $1<p<2$ and $2\leq p$,}
	replacing the operator $\E$ by $\E\left[\ \cdot\ \middle| \mathcal{F}_t\right]$ considering the BSDEs on $[t,T]$, which leads to the estimates $\E\left[\sup_{s\in{[t,T]}}|Y_s|^p \middle| \mathcal{F}_t\right]<C$ for all $t\in [0,T]$. The assertion now follows from the monotonicity of the conditional expectation.		
\end{proof}

\section{Proof of the Main \autoref{existence}} \label{sec:proof_of_lp_existence}

The proof basically follows the one in Briand et al. \cite[Theorem 4.2]{briand2003lp}. For convenience of the reader, we give a detailed proof adapted to our more general setting. We consider only the case $1<p< 2$ as the case $p \ge 2$ is similar but easier.

\bigskip

\textbf{Step 1:} Uniqueness\\
\sloppy Assume we have another solution $(Y',Z',U')$. Then \autoref{supproppleq2} applied to the BSDE $(0,g)$ with $g(t,y,z,u)=f(t,y+Y',z+Z',u+U')-f(Y',Z',U')$ implies $(Y-Y',Z-Z',U-U')=(0,0,0)$.

\bigskip

\textbf{Step 2:} \\
In this step, we construct a first approximating sequence of generators for $f$ and show several estimates for the solution processes. Assume that $\xi, \Ifzero \in L^\infty$. As \ref{A3le2} is satisfied, the condition is also satisfied for the changed parameter $\mu' = \rho(1)\alpha+\mu$. We take the constant $C 
$ appearing in \autoref{prop:stability} and choose an $r>C$.\bigskip

Take a smooth real function $\theta_r$ such that $0\leq \theta_r\leq 1$, $\theta_r(y)=1$ for $|y|\leq r$ and $\theta_r(y)=0$ for $|y|\geq r+1$ and define
$$h_n(t,y,z,u):=\theta_r(y)\left(f(t,y,c_n(z),\tilde{c}_n(u))-\fzerot\right)\frac{n}{\psi_{r+1}(t)\vee n}+\fzerot.$$
Here, $c_n, \tilde{c}_n$ are the projections $x \mapsto nx/(|x| \vee n)$ onto the closed unit balls of radius $n$, respectively in {$\Rdk$} and $L^2(\nu)$.

These generators $h_n$ satisfy the following properties for all $n \in \N$:
\begin{enumerate}[label={(A\,\roman*)}]
\item\label{i} Condition \ref{A1} is satisfied.

\item
By \ref{A2},
$$|h_n(t,y,z,u)|\leq n + |f_0(t)| + \Phi(t)(|z|+\|u\|).$$
\item\label{iii} 
By \ref{A2} and \ref{A3le2}, with $\beta = \beta_1 + \beta_2$, and $C_r$ denoting the Lipschitz constant of $\theta_r$, it holds that
\begin{align*}
&\innerproduct{y-y'}{h_n(t,y,z,u)-h_n(t,y',z',u')}\\
&=\theta_r(y)\frac{n}{\psi_{r+1}(t)\vee n}\innerproduct{y-y'}{f(t,y,c_n(z),\tilde{c}_n(u))-f(t,y',c_n(z'),\tilde{c}_n(u'))}\\
&\quad+\frac{n}{\psi_{r+1}(t)\vee n}\left(\theta_r(y)-\theta_r(y')\right)\innerproduct{y-y'}{f(t,y',c_n(z'),\tilde{c}_n(u'))-\fzerot}\\
&\leq \frac{n}{\psi_{r+1}(t)\vee n}\left(\alpha(t)\frac{\rho(|y-y'|^p)}{|y-y'|^{p-2}}+\mu(s)|y-y'|^2+\beta(t)(|z-z'|+\|u-u'\|)\right)\\
&\quad+C_r(2\Phi(s)+1)n|y-y'|^2\\
&\leq \alpha(t)\rho(|y-y'|^2)+\left(\rho(1)\alpha(s)+\mu(s)+C_r(2\Phi(s)+1)n\right)|y-y'|^2\\
&\quad+\beta(t)(|z-z'|+\|u-u'\|),
\end{align*}
where we used \autoref{remA3}\ref{remA3:A3bounds}\ref{remA3le2}. 
\item\label{iv} By \ref{A3le2}, again with $\beta = \beta_1 + \beta_2$ we have,
$$yh_n(t,y,z,u)\leq |y||\fzerot|+\alpha(t)\rho(|y|^2)+(\rho(1)\alpha(t)+\mu(t))|y|^2+\beta(t)|y|(|z|+\|u\|).$$
\end{enumerate}

\changed{Properties \ref{i}-\ref{iii} imply that for $\mu_r(s):=\left(\rho(1)\alpha(s)+\mu(s)+C_r(2\Phi(s)+1)n\right)$ the generator $g_n(t,y,z,u):=e^{\int_0^\cdot \mu_r(s)ds}\left(h_n-\mu_r\cdot y\right)$ satisfies assumptions (A\ 1)-(A\ 3) of Theorem 3.1 in \cite{geiss2018monotonic} (or rather a straightforward adaptation to $d$ dimensions of it) and admits a unique solution of BSDE $(e^{\int_0^T \mu_r(s)ds}\xi,g_n)$. Thus, by the transformation $(\tilde{Y},\tilde{Z},\tilde{U}):=e^{-\int_0^\cdot \mu_r(s)ds}(Y,Z,U)$, one gets that also $(\xi,h_n)$ has a unique solution $(Y^n,Z^n,U^n)$.}

\changed{Moreover, by property \ref{iv} and  \eqref{rem:weaker_assumption_on_monotonicity_1} we are} able to apply \autoref{prop:stability} to get that $\|Y_t^n\|_\infty\leq r$. Since $Y^n_t$ is bounded by $r$, we get that $(Y^n,Z^n,U^n)$ is also a solution to the BSDE $(\xi,f_n)$, with 
$$f_n(t,y,z,u):=\left(f(t,y,c_n(z),\tilde{c}_n(u))-\fzerot\right)\frac{n}{\psi_{r+1}(t)\vee n}+\fzerot .$$ 

Comparing the solutions $(Y^n,Z^n,U^n)$ and $(Y^m,Z^m,U^m)$ for $m\geq n$, we use the standard methods from \eqref{gammaeq1}-\eqref{gammaeq2}, for the differences $$(\Delta Y, \Delta Z, \Delta U):=(Y^m,Z^m,U^m)-(Y^n,Z^n,U^n).$$ In this procedure, we replace the use of the monotonicity condition \ref{A3ge2} in \autoref{supproppleq2} by
\begin{align*}
&|\Delta Y_s|^{p-2} \innerproduct{\Delta Y_s}{f_m(s,Y^m_s,Z^m_s,U^m_s)-f_n(s, Y^n_s,Z^n_s,U^n_s)}\nonumber\\
&= |\Delta Y_s|^{p-2} \innerproduct{\Delta Y_s}{f_m(s,Y^m_s,Z^m_s,U^m_s)-f_m(s, Y^n_s,Z^n_s,U^n_s)}\nonumber \\
&\quad+ |\Delta Y_s|^{p-2} \innerproduct{\Delta Y_s}{f_m(s,Y^n_s,Z^n_s,U^n_s)-f_n(s, Y^n_s,Z^n_s,U^n_s)} \nonumber\\
&\le \alpha(s)\rho(|\Delta Y_s|^p)+\left(\rho(1)\alpha(s)+\mu(s)+\frac{R_z+R_u}{2}\beta(s)^2\right)|\Delta Y_s|^p+\left(\frac{|\Delta Z_s|^2}{2R_z}+\frac{\|\Delta U_s\|^2}{2R_u}\right)\nonumber\\
&\quad+ |\Delta Y_s|^{p-2} \left| \innerproduct{\Delta Y_s}{f_m(s,Y^n_s,Z^n_s,U^n_s)-f_n(s, Y^n_s,Z^n_s,U^n_s)}\right|,
\end{align*}

such that  the same steps of the proof of \autoref{supproppleq2} can be conducted to get a function $h$ with
\begin{align*}
&\|\Delta Y\|^p_{\mathcal{S}^p} +\left\|\Delta Z\right\|_{L^p(W) }^p + \left\|\Delta U\right\|_{L^p(\tilde N) }^p \\
&\leq h\left(\E\int_0^T 
|\Delta Y_s|^{p-2} \left|\innerproduct{\Delta Y_s}{f_m(s,Y^n_s,Z^n_s,U^n_s)-f_n(s,Y^n_s,Z^n_s,U^n_s)}\right|ds\right)
\end{align*}
(in the case for \ref{A3ge2}, we use the steps from \autoref{supprop} and \autoref{YdeterminesZU}).
So $\|\Delta Y\|^p_{\mathcal{S}^p} +\left\|\Delta Z\right\|_{L^p(W) }^p + \left\|\Delta U\right\|_{L^p(\tilde N) }^p$ tends to zero if 
$$
\E\int_0^T |\Delta Y_s|^{p-2} \left|\innerproduct{\Delta Y_s}{f_m(s,Y^n_s,Z^n_s,U^n_s)-f_n(s, Y^n_s,Z^n_s,U^n_s)}\right|ds
$$
does, which we will show next (in the case of \ref{A3ge2}, this follows from \autoref{supprop} \changed{with  \eqref{rem:weaker_assumption_on_monotonicity_2}}).\bigskip

Since $|Y^m_t|,|Y^n_t|\leq r$, we estimate
\begin{align}\label{fnfmbound}
&\E\int_0^T |\Delta Y_s|^{p-2} \left| \innerproduct{\Delta Y_s}{f_m(s,Y^n_s,Z^n_s,U^n_s)-f_n(s, Y^n_s,Z^n_s,U^n_s)}\right|ds\nonumber\\
&\leq (2r)^{p-1}\E\int_0^T \left|f_m(s,Y^n_s,Z^n_s,U^n_s)-f_n(s, Y^n_s,Z^n_s,U^n_s)\right|ds.
\end{align}
\changed{Because of the definition of $f_m, f_n$ and since $m \ge n$, the integrand is zero if $|Z_s|\leq n, \|U_s\|\leq n$ and $\psi_{r+1}(s)\leq n$ 
and bounded by} 
\begin{align}\label{psibd}
&2\Phi(s)\left(|Z^n_s|+\|U^n_s\|\right)\chi_{\{|Z^n_s|+\|U^n_s\|>n\}}+2\Phi(s)\left(|Z^n_s|+\|U^n_s\|\right)\chi_{\{\psi_{r+1}(s)>n\}}\nonumber\\
&+2\psi_{r+1}(s)\chi_{\{\psi_{r+1}(s)>n\}}+2\psi_{r+1}(s)\chi_{\{|Z^n_s|+\|U^n_s\|>n\}}
\end{align}
otherwise. 

To show convergence of the integral of \eqref{psibd}, we use the uniform integrability of the families $\Phi(|Z^n|+\|U^n\|)_{n\geq 1}$ with respect to the measure $\mathbb{P}\otimes\lambda$, which follows from 
\begin{align*}
&\E\int_0^T\Phi(s)(|Z^n_s|+\|U^n_s\|)ds\leq \E\int_0^T \left(\Phi(s)^2+|Z^n_s|^2+\|U^n_s\|^2\right)ds\\
&\leq \left\|\int_0^T \Phi(s)^2ds\right\|_\infty +r',
\end{align*}
since by \autoref{supproppleq2} and \ref{iv}, there is $r'>0$ such that $\left\|Z^n\right\|_{L^2(W) }^2 + \left\| U^n\right\|_{L^2(\tilde N) }^2 < r'$.
Therefore, as \eqref{psibd} (as sequence in $n$) is uniformly integrable with respect to $\mathbb{P}\otimes\lambda$, dominating the sequence  $\left(\left|f_m(s,Y^n_s,Z^n_s,U^n_s)-f_n(s, Y^n_s,Z^n_s,U^n_s)\right|\right)_{n\geq 0}$, which approaches zero pointwisely, also \eqref{fnfmbound} tends to zero as $m>n\to\infty$.
Hence, also $\|\Delta Y\|^p_{\mathcal{S}^p} +\left\|\Delta Z\right\|_{L^p(W) }^p + \left\|\Delta U\right\|_{L^p(\tilde N) }^p$ tend to zero, showing that the $(Y^n,Z^n,U^n)$ form a Cauchy sequence in $\mathcal{S}^p\times L^p(W)\times L^p(\tilde{N})$ and converge to an element $(Y,Z,U)$.
\bigskip

\textbf{Step 3:}\\
\changed{We now show that $(Y,Z,U)$ satisfies the BSDE $(\xi, f)$, for $\xi, \Ifzero \in L^\infty$, as supposed in step 2}.
The stochastic integral terms of the BSDEs $(\xi,f_n)$ with solution $(Y^n,Z^n,U^n)$ converge to the corresponding terms of the BSDE $(\xi,f)$ also in probability. It is left to show that, at least for a subsequence,
$$\int_t^T f_n(s,Y^n_s,Z^n_s,U^n_s)ds\to\int_t^T f(s,Y_s,Z_s,U_s)ds,\quad\mathbb{P}\text{-a.s.}$$
For an appropriate subsequence all other terms of the BSDEs converge almost surely. W.l.o.g, this subsequence is assumed to be the original one. Hence, we know that there is a random variable $V_t$ such that 
$$\int_t^T f_n(s,Y^n_s,Z^n_s,U^n_s)ds\to V_t,\quad\mathbb{P}\text{-a.s.}$$
We take expectations and split up the integral into
$$\delta^{(1)}:=\E\int_t^T \left(f_n(s,Y^n_s,Z^n_s,U^n_s)-f(s,Y^n_s,Z^n_s,U^n_s)\right)ds$$
and
$$\delta^{(2)}:=\E\int_t^T \left(f(s,Y^n_s,Z^n_s,U^n_s)-f(s,Y_s,Z_s,U_s)\right)ds.$$
By the same argument as for inequality \eqref{psibd} above, 
\begin{align*}
|\delta^{(1)}|&\leq \E\int_0^T\biggl[2\Phi(s)\left(|Z^n_s|+\|U^n_s\|\right)\chi_{\{|Z^n_s|+\|U^n_s\|>n\}}\\
&+2\Phi(s)\!\left(|Z^n_s|+\|U^n_s\|\right)\!\chi_{\{\psi_{r+1}(s)>n\}}\!+2\psi_{r+1}(s)\chi_{\{\psi_{r+1}(s)>n\}}\!+2\psi_{r+1}(s)\chi_{\{|Z^n_s|+\|U^n_s\|>n\}}\biggr]ds,
\end{align*}
which converges to zero. Now, for $\delta^{(2)}$, \changed{we know that $(Y^n)\to Y$ in $\mathcal{S}^p$, hence there is a subsequence $(Y^{n_\ell})_\ell\geq 1$ s.t. $(Y^{n_\ell}\to Y)$ uniformly in $t$, $\mathbb{P}$-a.s. Since all $Y^{n_l}$ are c\`adl\`ag (as $Y$-components of solutions to BSDEs) and bounded by $r$ from step 2, we also get that the limit $Y$ is c\`adl\`ag and $\|Y_t\|_\infty$ is bounded by $r$ for all $t\in{[0,T]}$. Further, we know that 
$(Y^n,Z^n,U^n)\to(Y,Z,U)$ in the measure $\mathbb{P}\otimes\lambda$. Thus also 
$f(s,Y^n_s,Z^n_s,U^n_s)\to f(s,Y_s,Z_s,U_s)$ in $\mathbb{P}\otimes\lambda$. Since now 
\begin{align*}
\delta^{(2)}&=\E\int_t^T \left(f(s,Y^n_s,Z^n_s,U^n_s)-f(s,Y_s,Z_s,U_s)\right)ds\\
&=\E\int_t^T \left(f(s,Y^n_s,Z^n_s,U^n_s)-f_0(s)-(f(s,Y_s,Z_s,U_s)-f_0(s))\right)ds
\end{align*} 
and, by \ref{A2},
\begin{align*}
&|f(s,Y^n_s,Z^n_s,U^n_s)-f_0(s)|+|f(s,Y_s,Z_s,U_s)-f_0(s)|\\
&\leq 2\psi_{r+1}(s)+\Phi(s)\left(|Z^n_s|+|Z_s|+\|U^n_s\|+\|U_s\|\right),\end{align*}
from the uniform integrability of $\left(\psi_{r+1}+\Phi(|Z^n|+|Z|+\|U^n\|+\|U\|)\right)_{n\geq 1}$ with respect to $\mathbb{P}\otimes\lambda$, it follows that also $|\delta^{(2)}|\to 0$.}

Thus, $$\E\int_t^T f_n(s,Y^n_s,Z^n_s,U^n_s)ds\to\E\int_t^T f(s,Y_s,Z_s,U_s)ds,$$
and extracting a subsequence $(n_l)_{l\geq 1}$ satisfying $\mathbb{P}$-a.s
$$\int_t^T f_{n_l}(s,Y^{n_l}_s,Z^{n_l}_s,U^{n_l}_s)ds\to\int_t^T f(s,Y_s,Z_s,U_s)ds,$$
shows that $V_t=\int_t^T f(s,Y_s,Z_s,U_s)ds$. So $(Y,Z,U)$ satisfies the BSDE $(\xi,f)$.
\bigskip

\textbf{Step 4:}\\
We now approximate a general $\xi\in L^p$ by $c_n(\xi)$ and the generator 
$f$ by 
$$
f^n(t,y,z,u):=f(t,y,z,u)-\fzerot+c_n(\fzerot).
$$ 
A solution $(Y^n,Z^n,U^n)$ to $(c_n(\xi),f^n)$ exists due to the last step. Now we get, for $m\geq n$, denoting differences again by
$$(\Delta Y, \Delta Z, \Delta U):=(Y^m,Z^m,U^m)-(Y^n,Z^n,U^n)$$
via \autoref{supproppleq2} (we use the generator $g_{m,n}(t,y,z,u):=f^m(t,y+Y^n,z+Z^n,u+U^n)-f^n(Y^n,Z^n,U^n)$):
\begin{align*}
&\|\Delta Y\|^p_{\mathcal{S}^p} +\left\|\Delta Z\right\|_{L^2(W) }^2 + \left\|\Delta U\right\|_{L^2(\tilde N) }^2 \\
&\leq h\Biggl(\E|c_m(\xi)-c_n(\xi)|^p+\E\left(\int_0^T \left|f^m(s,Y^n_s,Z^n_s,U^n_s)-f^n(Y^n_s,Z^n_s,U^n_s)\right|ds\right)^{\!\!p}\Biggr)\\
&=h\left(\E|c_m(\xi)-c_n(\xi)|^p+\E\left(\int_0^T \left|c_m(f_0(s))-c_n(f_0(s))\right|ds\right)^{\!\!p}\right).
\end{align*}
As $n\to\infty$, the latter term tends to zero showing convergence of the sequence $(Y^n,Z^n,U^n)$ to $(Y,Z,U)$ in ${\mathcal{S}^p}\times{L^p(W) }\times L^p(\tilde N)$.
Again it remains to check that
$$\int_t^T f^n(s,Y^n_s,Z^n_s,U^n_s)ds\to\int_t^T f(s,Y_s,Z_s,U_s)ds, \quad \mathbb{P}\text{-a.s.,}$$
at least for a subsequence. \changed{This is equivalent to,
	$$\int_t^T \left(f^n(s,Y^n_s,Z^n_s,U^n_s)-f(s,Y^n_s,Z^n_s,U^n_s)\right)ds- \int_t^T \left(f(s,Y_s,Z_s,U_s)-f(s,Y^n_s,Z^n_s,U^n_s)\right)ds \to 0, \quad \mathbb{P}\text{-a.s.,}$$

The first integral is
\begin{align*}
\int_t^T \left(f^n(s,Y^n_s,Z^n_s,U^n_s)-f(s,Y^n_s,Z^n_s,U^n_s)\right)ds
=\int_t^T \left(c_n(f_0(s))-f_0(s)\right)ds,
\end{align*}
which tends to zero as $n\to\infty$.
\bigskip

The second integral is more complicated.}
Here, we extract a subsequence $(n_l)_{l\geq 1}$ such that $$\sup_{t\in[0,T]}|Y_t-Y^{n_l}_t|\to 0,\quad\mathbb{P}\text{-a.s.},\quad(Z^{n_l},U^{n_l})\to (Z,U),\quad\lambda\text{-a.e.},\mathbb{P}\text{-a.s.}$$ and
$$\left(\int_0^T |Z^{n_l}_s|^2 ds,\int_0^T \|U^{n_l}_s\|^2 ds\right)\to\left(\int_0^T |Z_s|^2 ds,\int_0^T \|U_s\|^2 ds\right),\quad\mathbb{P}\text{-a.s.},$$
which is possible due to the convergence of $(Y^n,Z^n,U^n)$ in $\mathcal{S}^p\times L^p(W)\times L^p(\tilde{N})$.
Severini-Egorov's theorem now permits for a given $\varepsilon>0$ the existence of a set $\Omega_\varepsilon, \mathbb{P}(\Omega_\varepsilon)>1-\varepsilon$ such that there is a number $N_\varepsilon>0$ with $$\sup_{t\in[0,T]}|Y_t(\omega)-Y^{n_l}_t(\omega)|<1\quad \text{for all }l>N_\varepsilon\text{ and }\omega\in\Omega_\varepsilon,$$
and the convergences above persist on this set.
For given $r>0$ on $\Omega_\varepsilon^r:=\Omega_\varepsilon\cap\left\{\sup_{t\in[0,T]}|Y_t|\leq r\right\}$, we have for $l>N_\varepsilon$,
\begin{align*}
&\int_0^T |f(s,Y^{n_l}_s,Z^{n_l}_s,U^{n_l}_s)-f(s,Y_s,Z_s,U_s)|ds\\
&\leq \int_0^T \left(2\psi_{r+1}(s)+\Phi(s)\left(|Z^{n_l}_s|+|Z_s|+\|U^{n_l}_s\|+\|U_s\|\right)\right)ds<C<\infty,
\end{align*}
where $C$ may still depend on $\omega\in \Omega_\varepsilon^r$ and for such $\omega$, \changed{the family $\left(s\mapsto(|Z^{n_l}_s\left(\omega)|^2+\|U^{n_l}_s(\omega)\|^2\right)_{l\geq 0}\right)$ }is uniformly integrable with respect to $\lambda$.
Thus, since $f(s,Y^{n_l}_s,Z^{n_l}_s,U^{n_l}_s)\to f(s,Y_s,Z_s,U_s)$, for $\lambda$-a.a. $s\in [0,T]$ on $\Omega_\varepsilon^r$, dominated convergence yields 
$$\lim_{l\to\infty}\int_t^Tf(s,Y^{n_l}_s,Z^{n_l}_s,U^{n_l}_s)ds= \int_t^Tf(s,Y_s,Z_s,U_s)ds$$ 
for $l\to\infty$ on $\Omega_{\varepsilon}^r$ for all $\varepsilon>0,r>0$. The last identity now even holds on $\Omega^0:=\bigcup_{r>0}\bigcup_{q\geq 1} \Omega_{1/q}^r$ which is an almost sure event.
So, the limit of $\int_t^T f^n(s,Y^n_s,Z^n_s,U^n_s)ds$ is uniquely determined as $\int_t^T f(s,Y_s,Z_s,U_s)ds$.
Hence, $(Y,Z,U)$ is a solution to BSDE $(\xi,f)$. \qed

\section{Comparison Results} \label{sec:comparison_result}

We switch to dimension $d = 1$ and set $\R_0 = \R \setminus \{0\}$ for the following comparison results, generalizing those in \cite{geiss2018monotonic} to the case of generators that do not have linear growth in the $y$-variable and to an $L^p$-setting for $p>1$. Moreover, in contrast to \cite{geiss2018monotonic}, our proof does not depend on approximation theorems for BSDEs that demand deep measurability results.

\begin{theorem}[Comparison, $p\geq 2$]\label{comparison}
Let $p, p' \ge 2$ and $(Y,Z,U)$ be the $\lp$-solution to $(\xi,f)$ and $(Y',Z',U')$ be the $L^{p'}$-solution to $(\xi',f')$. Furthermore let $f$ and $f'$ satisfy \ref{A1} and \ref{A3ge2} for the according $p,p'$. If the following assumptions hold
\begin{enumerate}[label=({\roman*})]
\item
 $\xi\leq \xi'$, $\mathbb{P}$-a.s.,
 \item
 $f(s,Y'_s,Z'_s,U'_s)\leq f'(s,Y'_s,Z'_s,U'_s)$, for $\mathbb{P}\otimes\lambda$-a.a. $(\omega,s)\in\Omega\times{[0,T]}$   
 and
	\item[(A$\gamma$)] for all $u,u'\in L^2(\nu)$ with $u'\geq u$
	\begin{align}
	\label{gamma} 
	f(s,y,z,u)- f(s,y,z,u')\leq \int_{\mathbb{R}_0}(u'(x)-u(x))\nu(dx), \quad\mathbb{P}\otimes\lambda\text{-a.e},
	\end{align}
\end{enumerate}

then for all $t \in [0,T]$, we have $\mathbb{P}$-a.s, 
$$
Y_t\leq Y'_t.
$$

The same assertion follows from an equivalent formulation for $f'$, requiring $f(s,Y_s,Z_s,U_s) \leq f'(s,Y_s,Z_s,U_s)$ and \eqref{gamma} being satisfied for $f'$.
\end{theorem}

\begin{proof}

The basic idea for this proof was inspired by the one of Theorem 8.3 in \cite{carmona} and is an extension and simplification of the one in \cite{geiss2018monotonic}.

\bigskip

\textbf{Step 1:}

First, note that $(Y,Z,U), (Y',Z',U')$ are solutions in $\mathcal{S}^2\times L^2(W)\times L^2(\tilde{N})$.

We use the conditional expectation $\E_n$ (see \autoref{sec:finlev}) on the BSDEs $(\xi,f)$ and $(\xi',f')$ to get (for the BSDE $(\xi,f)$)

\begin{align*}
\E_n Y_t=&\E_n\xi+\int_t^T \E_n f(s,Y_s,Z_s,U_s)ds-\int_t^T \E_n Z_s dW_s\\
& -\inttTR\chi_{\{1/n\leq |x|\}}\E_n U_s(x)\tilde{N}(ds,dx)\,.
\end{align*}

Note that here $(\E_n Z_t)_{t\in{[0,T]}}, (\E_n U_t(x))_{t\in{[0,T]}, x\in \R_0}$ are considered to be predictable (and $(\E_n f(t,Y_t,Z_t,U_t))_{t\in{[0,T]}}$ progressively measurable) processes that equal the conditional expectation $\mathbb{P}$-a.s. for almost all $t\in {[0,T]}$ (or $\lambda\otimes\nu$ almost every $(t,x)\in {[0,T]}\times\R_0$ for the $U$-process). In the case of $Y$, $(\E_n Y_t)_{t\in{[0,T]}},$ denotes a progressively measurable version of this process. For bounded or nonnegative processes, the construction of such processes can be achieved by using optional projections with parameters (see \cite{meyer} for optional projections with parameters, \cite{geiss2018monotonic} for the mentioned construction). In the present case, we are confronted with merely integrable processes: $Y, Z, \int_{\R_0}|U_\cdot(x)|^2\nu(dx)$ are integrable and hence also $f(s,Y_s,Z_s,U_s)\in L^1(W)$. The construction of a progressively measurable version for the processes at hand can be found in \autoref{measurable-version} and \autoref{measurable-version-remark} in the Appendix.
\bigskip

Moreover, assume for the rest of the proof that the coefficient $\mu$ of $f$ is zero: If this was not the case, we could use the transformed variables $(\tilde{Y}_t,\tilde{Z}_t,\tilde{U}_t):=e^{\int_0^t\mu(s)ds}(Y_t,Z_t,U_t)$ and $(\tilde{Y}',\tilde{Z}',\tilde{U}'):=e^{\int_0^t\mu(s)ds}(Y'_t,Z'_t,U'_t)$.
\bigskip

\textbf{Step 2:}
We use Tanaka-Meyer's formula \changed{(cf. \cite[Chapter 4, Theorem 70 and Corollary 1]{protter}) for the squared positive part function $(\cdot)_{+}^2$} to see that for $\eta:=18\beta^2$,\changed{
\begin{align*}
&e^{\int_0^t \eta(s)ds}(\E_nY_t-\E_nY'_t)^2_+=e^{\int_0^T \eta(s)ds}(\E_n\xi-\E_n\xi')^2_++M(t)\\
& +\int_t^T e^{\int_0^s \eta(\tau)d\tau}\\
&\quad\quad\times\biggl[2(\E_nY_s-\E_nY'_s)_+\E_n\left(f(s,Y_s,Z_s,U_s)-f'(s,Y'_s,Z'_s, U'_s)\right)\\
&\quad\quad\quad\ \ -\chi_{\{\E_nY_s-\E_nY'_s > 0\}}|\E_nZ_s-\E_nZ'_s|^2- \eta(s)(\E_nY_s-\E_nY'_s)_+^2\\
&\quad\quad\quad\quad-\int_{\{1/n\leq|x|\}}\left((\E_nY_s-\E_nY'_s+\E_nU_s(x)-\E_nU'_s(x))^2_+-(\E_nY_s-\E_nY'_s)^2_+\right.\\
&\quad\quad\quad\quad\quad\quad\quad\quad\quad\quad\quad\quad\quad\quad\quad\quad\left.-2(\E_nU_s(x)-\E_nU'_s(x))(\E_nY_s-\E_nY'_s)_+\right)\nu(dx)\biggr]ds.
\end{align*}
Here, $M(t)$ is a stochastic integral term with zero expectation which follows from  $Y,Y'\in\mathcal{S}^2$. Moreover, we used that on the set $\{\Delta^n Y_s > 0\}$ (where  $\Delta^n Y:=\E_n Y-\E_n Y'$) we have 
$(Y_s-Y'_s)_+=|Y_s-Y'_s|$. We denote further differences by  $\Delta^n \xi := \E_n\xi-\E_n\xi', \,\Delta^n Z := \E_nZ-\E_nZ'$ and $\Delta^n U:= \E_nU-\E_nU'$. Observing that the right hand side increases if we only consider $\Delta^n Y_s > 0$ in the $ds$-integral, we take means to come to

\begin{align*}
&\E e^{\int_0^t \eta(s)ds}(\Delta^n Y_t)^2_+=\E e^{\int_0^T \eta(s)ds}(\Delta^n \xi)^2_+ \nonumber\\
&+\E\int_t^T e^{\int_0^s \eta(\tau)d\tau}\chi_{\{\Delta^n Y_s > 0\}}\nonumber\\
&\quad\quad\quad\times\biggl[2(\Delta^n Y_s)_+\E_n\left(f(s,Y_s,Z_s,U_s)-f'(s,Y'_s,Z'_s, U'_s)\right)   -|\Delta^n Z_s|^2- \eta(s)|\Delta^n Y_s|^2\nonumber\\
&\quad\quad\quad\quad\ \ - \int_{\{1/n\leq|x|\}}\left((\Delta^n Y_s+ \Delta^n U_s(x))^2_+-(\Delta^n Y_s)^2_+-2(\Delta^n U_s(x))(\Delta^n Y_s)_+\right)\nu(dx)\biggr]ds,
\end{align*} 
We split up the set $\{1/n\leq|x|\}$ into 
$$B_n(s) = B_n(\omega,s) =
\{1/n\leq|x|\}\cap\{\Delta^n U_s(x)\geq  -\Delta^n Y_s\}\text{ and its complement } B_n^c(s).$$

Taking into account that $\xi\leq \xi'\Rightarrow \E_n\xi\leq\E_n \xi'$ we estimate
\begin{align}\label{TanMeyf}
&\E e^{\int_0^t \eta(s)ds}(\Delta^n Y_t)^2_+\nonumber\\
&\leq\E\int_t^T e^{\int_0^s \eta(\tau)d\tau}\chi_{\{\Delta^n Y_s> 0\}}\nonumber\\
&\quad\quad\quad\times\!\biggl[2(\Delta^n Y_s)_+\E_n\left(f(s,Y_s,Z_s,U_s)-f'(s,Y'_s,Z'_s, U'_s)\right) -|\Delta^n Z_s|^2- \eta(s)|\Delta^n Y_s|^2\nonumber\\
& \quad\quad\quad\quad\ -\!\int_{B_n(s)}\!|\Delta^n U_s(x)|^2\nu(dx)+\int_{B_n^c(s)}\!\left((\Delta^n Y_s)^2_++2(\Delta^n U_s(x))(\Delta^n Y_s)_+\right)\nu(dx)\biggr]ds.
\end{align} 
}
We focus on $(\Delta^n Y_s)_+\E_n\left(f(s,Y_s,Z_s,U_s)-f'(s,Y'_s,Z'_s, U'_s)\right)$. By abbreviating $\Theta := (Y,Z,U)$, $\Theta' := (Y',Z',U')$ and by the assumption $f(s,\Theta') \le f'(s,\Theta')$ we get,
\begin{align}\label{flipdiff}
& (\Delta^n Y_s)_+\E_n\left(f(s,\Theta_s)-f'(s,\Theta'_s)\right) =(\Delta^n Y_s)_+\E_n\left(f(s,\Theta_s)-f(s,\Theta'_s)+f(s,\Theta'_s)-f'(s,\Theta'_s)\right)\nonumber\\
&\leq(\Delta^n Y_s)_+\E_n\left(f(s,\Theta_s)-f(s,\Theta'_s)\right).
\end{align}
Since \ref{A3ge2} implies the Lipschitz property in the $u$ and $z$-variables, we infer, inserting and subtracting the same terms,
\begin{align}\label{flip}
& (\Delta^n Y_s)_+\E_n\left(f(s,Y_s,Z_s,U_s)-f(s,Y'_s,Z'_s, U'_s)\right)\\
& =\!(\Delta^n Y_s)_+\E_n\Big(\!f(s,Y_s,Z_s,U_s)\!-\!f(s,Y'_s,Z'_s, U'_s)\!+\!\left(f(s,Y_s,\E_n Z_s,\E_n U_s)\!-\!f(s, Y'_s,\E_n Z'_s, \E_n U'_s)\right)\nonumber\\
&\quad-\left(f(s,Y_s,\E_n Z_s,\E_n U_s)-f(s, Y'_s,\E_n Z'_s, \E_n U'_s)\right)\Big)\nonumber\\
&\leq(\Delta^n Y_s)_+\E_n\left(f(s,Y_s,\E_n Z_s,\E_n U_s)-f(s,Y_s',\E_n Z'_s, \E_n U'_s)\right)\nonumber\\
&\quad+(\Delta^n Y_s)_+\beta(s)\left(|Z_s-\E_n Z_s|+|Z'_s-\E_n Z'_s|+\|U_s-\E_n U_s\|+\|U'_s-\E_n U'_s\|\right).\nonumber
\end{align}

We estimate, inserting and subtracting terms again, then using \eqref{gamma},
\begin{align}\label{flipgam}
&(\Delta^n Y_s)_+\E_n\left(f(s,Y_s,\E_n Z_s,\E_n U_s)-f(s,Y_s',\E_n Z'_s, \E_n U'_s)\right)\nonumber\\
&\leq  (\Delta^n Y_s)_+\E_n\left(f(s,Y_s,\E_n Z_s,\E_n U_s\chi_{B_n(s)}+\E_n U_s\chi_{B^c_n(s)})\right.\nonumber\\
&\quad\quad\quad\quad\quad\quad\quad\quad\quad\quad\quad\quad\quad\quad\quad\left.-\!f(s,Y_s',\E_n Z'_s, \E_n U'_s\chi_{B_n(s)}+\E_n U_s\chi_{B_n^c(s)})\right)_+\nonumber\\
& \quad\!+\!(\Delta^n Y_s)_+\E_n\left(f(s,Y_s,\E_n Z_s,\E_n U'_s\chi_{B_n(s)}\!+\E_n U_s\chi_{B^c_n(s)})\!\right.\nonumber\\
&\quad\quad\quad\quad\quad\quad\quad\quad\quad\quad\quad\quad\quad\quad\quad\left.-\!f(s,Y_s',\E_n Z'_s, \E_n U'_s\chi_{B_n(s)}\!+\E_n U'_s\chi_{B_n^c(s)})\right)\nonumber\\
&\leq (\Delta^n Y_s)_+\E_n\left(f(s,Y_s,\E_n Z_s,\E_n U_s\chi_{B_n(s)}+\E_n U_s\chi_{B^c_n(s)})\right.\nonumber\\
&\quad\quad\quad\quad\quad\quad\quad\quad\quad\quad\quad\quad\quad\quad\quad\left.-\!f(s,Y_s',\E_n Z'_s, \E_n U'_s\chi_{B_n(s)}+\E_n U_s\chi_{B_n^c(s)})\right)_+\nonumber\\
&\quad - \int_{B_n^c(s)}(\Delta^n Y_s)_+\Delta^n U_s\nu(dx).
\end{align}

Next we apply Jensen's inequality in two dimensions for the product of positive random variables and also \ref{A3ge2} and Young's inequality to arrive at
\begin{align}\label{flipa3}
&(\Delta^n Y_s)_+\E_n\left(f(s,Y_s,\E_n Z_s,\E_n U_s\chi_{B_n(s)}+\E_n U_s\chi_{B^c_n(s)})\right.\nonumber\\
&\quad\quad\quad\quad\quad\quad\quad\quad\quad\quad\quad\quad\quad\quad\quad\left.-f(s,Y_s',\E_n Z'_s, \E_n U'_s\chi_{B_n(s)}+\E_n U_s\chi_{B_n^c(s)})\right)_+\nonumber\\
&\leq \E_n\left[\Delta Y_s \left(f(s,Y_s,\E_n Z_s,\E_n U_s\chi_{B_n(s)}+\E_n U_s\chi_{B^c_n(s)})\right.\right.\nonumber\\
&\quad\quad\quad\quad\quad\quad\quad\quad\quad\quad\quad\quad\quad\quad\quad\left.\left.-f(s,Y_s',\E_n Z'_s, \E_n U'_s\chi_{B_n(s)}+\E_n U_s\chi_{B_n^c(s)})\right)\right]_+\nonumber\\
&\leq \E_n\alpha(s)\rho((\Delta Y_s)_+^2)+2\E_n\beta(s)^2(\Delta Y_s)_+^2+\frac{|\Delta^n Z_s|^2}{4}+\frac{\E_n\|\Delta^n U_s\chi_{B_n(s)}\|^2}{4}.
\end{align}

Taking together inequalities \eqref{flipdiff}, \eqref{flip}, \eqref{flipgam} and \eqref{flipa3},
we get with Young's inequality again that
\begin{align*}
& (\Delta^n Y_s)_+\E_n\left(f(s,Y_s,Z_s,U_s)-f'(s,Y'_s,Z'_s, U'_s)\right)\\
&\leq \E_n\alpha(s)\rho((\Delta Y_s)_+^2)+2\E_n\beta(s)^2(\Delta Y_s)_+^2+\frac{|\Delta^n Z_s|^2}{4}+\frac{\E_n\|\Delta^n U_s\chi_{B_n(s)}\|^2}{4}\nonumber\\
&\quad -\int_{B_n^c(s)}(\Delta^n Y_s)_+\Delta^n U_s\nu(dx)\\
&+4\beta(s)^2(\Delta^n Y_s)^2_++\frac{1}{4}\left(|Z_s-\E_n Z_s|^2+|Z'_s-\E_n Z'_s|^2+\|U_s-\E_n U_s\|^2+\|U'_s-\E_n U'_s\|^2\right).
\end{align*}
Therefore, \eqref{TanMeyf} evolves to 
\begin{align*}
&\E e^{\int_0^t \eta(s)ds}(\Delta^n Y_t)^2_+\\
&\leq \ \E\int_t^T e^{\int_0^s  \eta (\tau)d\tau}\chi_{\{\Delta^n Y_s> 0\}}\biggl[2\E_n\alpha(s){\rho}((\Delta Y_s)_+^2)+4\E_n\beta(s)^2(\Delta Y_s)_+^2+\frac{|\Delta^n Z_s|^2}{2}\\
& \quad+\frac{\|\Delta^n U_s\chi_{B_n(s)}\|^2}{2} -\int_{B_n^c(s)}2(\Delta^n Y_s)_+ \Delta^n U_s(x)\nu(dx)+8\beta(s)^2(\Delta^n Y_s)^2_+\\
&\quad+\frac{1}{2}\left(|Z_s-\E_n Z_s|^2+|Z'_s-\E_n Z'_s|^2+\|U_s-\E_n U_s\|^2+\|U'_s-\E_n U'_s\|^2\right)-|\Delta^n Z_s|^2\\
&\quad-\eta(s)|\Delta^n Y_s|^2-\!\!\int_{B_n(s)}\!|\Delta^n U_s(x)|^2\nu(dx)+\!\int_{_{B_n^c(s)}}\!\!\!\left((\Delta^n Y_s)^2_++2(\Delta^n Y_s)_+(\Delta^n U_s(x))\right)\nu(dx)\biggr]ds.\nonumber
\end{align*} 
We cancel out terms and end this step with the estimate\changed{
\begin{align}\label{estimate-conditionaloutside}
&\E e^{\int_0^t \eta(s)ds}(\Delta^n Y_t)^2_+\nonumber\\
&\leq \ \E\int_t^T e^{\int_0^s  \eta (\tau)d\tau}\chi_{\{\Delta^n Y_s> 0\}}\biggl[2\E_n\alpha(s){\rho}((\Delta Y_s)_+^2)+4\E_n\beta(s)^2(\Delta Y_s)_+^2\nonumber\\
&\quad+\frac{1}{2}\left(|Z_s-\E_n Z_s|^2+|Z'_s-\E_n Z'_s|^2+\|U_s-\E_n U_s\|^2+\|U'_s-\E_n U'_s\|^2\right)\nonumber\\
&\quad+8\beta(s)^2(\Delta^n Y_s)^2_+-\eta(s)|\Delta^n Y_s|^2+\int_{_{B_n^c(s)}}(\Delta^n Y_s)^2_+\nu(dx)\biggr]ds.
\end{align} }

\textbf{Step 3:}

We assume without loss of generality, that the integrals 
\begin{align*}
\E\int_0^Te^{\int_0^s  \eta (\tau)d\tau}\chi_{\{\Delta Y_s> 0\}}\alpha(s){\rho}((\Delta Y_s)_+^2)ds=:\delta_\rho
\end{align*}
and
\begin{align*} 
&\E\int_0^Te^{\int_0^s  \eta (\tau)d\tau}\chi_{\{\Delta Y_s> 0\}}\beta(s)^2(\Delta Y_s)_+^2ds=:\delta_y \,.
\end{align*} 
are positive numbers. All other cases would simplify the proof.

Since $\E_n Y_s\to Y_s$ a.s. for all $s$, dominated convergence shows that also
$$\E\int_0^Te^{\int_0^s  \eta (\tau)d\tau}\left|\chi_{\{\Delta^n Y_s> 0\}}\alpha(s){\rho}((\Delta Y_s)_+^2)-\chi_{\{\Delta^n Y_s> 0\}}\alpha(s){\rho}((\Delta^n Y_s)_+^2)\right|ds,$$
$$\E\int_0^Te^{\int_0^s  \eta (\tau)d\tau}\left|\chi_{\{\Delta Y_s> 0\}}\alpha(s){\rho}((\Delta Y_s)_+^2)-\chi_{\{\Delta^n Y_s> 0\}}\alpha(s){\rho}((\Delta^n Y_s)_+^2)\right|ds$$
and 
$$\E\int_0^Te^{\int_0^s  \eta (\tau)d\tau}\left|\chi_{\{\Delta Y_s> 0\}}\alpha(s){\rho}((\Delta Y_s)_+^2)-\chi_{\{\Delta^n Y_s> 0\}}\alpha(s){\rho}((\Delta Y_s)_+^2)\right|ds$$
converge to zero. For domination we use
\begin{align*}
&\E\int_0^Te^{\int_0^s  \eta (\tau)d\tau}\alpha(s)\left({\rho}((\Delta Y_s)_+^2)+{\rho}\Big(\sup_{n\geq 0}\E_n(\Delta Y_s)_+^2\Big)\right)ds\\
&\leq e^C\|\alpha\|_{L^1([0,T])}(1+b\|Y\|^2_{\mathcal{S}^2})\\
&\quad+e^C\|\alpha\|_{L^1([0,T])}\left((1+b)(\sup_{n\geq 0}\E_n(\sup_{t\in{[0,T]}}\Delta Y_t))^2\right)\\
&\leq e^C\|\alpha\|_{L^1([0,T])}(2+5b\|Y\|^2_{\mathcal{S}^2})<\infty,
\end{align*}
where we applied that $\int_0^T \eta(s)ds<C$ a.s., Doob's martingale inequality and that there is $b>0$ such that for $x\geq 0: \rho(x)\leq 1+b x$.

For $m\geq 0$ with $\delta_\rho-\frac{1}{m}>0$ let us now choose $N_m\in\mathbb{N}$ large enough, such that for $n\geq N_m$:
$$\E\int_0^Te^{\int_0^s  \eta (\tau)d\tau}\left|\chi_{\{\Delta^n Y_s> 0\}}\alpha(s){\rho}((\Delta Y_s)_+^2)-\chi_{\{\Delta^n Y_s> 0\}}\alpha(s){\rho}((\Delta^n Y_s)_+^2)\right|ds<\delta_\rho-\frac{1}{m}$$ and
$$\E\int_0^Te^{\int_0^s  \eta (\tau)d\tau}\chi_{\{\Delta^n Y_s> 0\}}\alpha(s){\rho}((\Delta^n Y_s)_+^2)ds\geq \delta_\rho-\frac{1}{m}.$$
For such an $n$ we get that 
\begin{align*}
&\E\int_0^Te^{\int_0^s  \eta (\tau)d\tau}\chi_{\{\Delta^n Y_s> 0\}}\alpha(s){\rho}((\Delta Y_s)_+^2)ds\\
&\leq 2\E\int_0^Te^{\int_0^s  \eta (\tau)d\tau}\chi_{\{\Delta^n Y_s> 0\}}\alpha(s){\rho}((\Delta^n Y_s)_+^2)ds.
\end{align*}
In the same way, one can choose $m,N_m\in\mathbb{N}$ also large enough such that for all $n\geq N_m$,
\begin{align}\label{anti-jensen}
&\E\int_0^Te^{\int_0^s  \eta (\tau)d\tau}\chi_{\{\Delta^n Y_s> 0\}}\beta(s)^2(\Delta Y_s)_+^2ds\nonumber\\
&\quad\quad\leq 2\E\int_0^Te^{\int_0^s  \eta (\tau)d\tau}\chi_{\{\Delta^n Y_s> 0\}}\beta(s)^2(\Delta^n Y_s)_+^2ds.
\end{align}
Similarly, by martingale convergence $\E_n Z_s\to Z_s$ and a domination argument, we can conclude that for $n\geq N_m$ ($N_m$ may have to be rechosen large enough),
\begin{align}\label{anti-jensen2}
&\E\int_0^Te^{\int_0^s  \eta (\tau)d\tau}\chi_{\{\Delta^n Y_s> 0\}}| Z_s-\E_nZ_s|^2ds \leq \E\int_0^Te^{\int_0^s  \eta (\tau)d\tau}\chi_{\{\Delta^n Y_s> 0\}}\beta(s)^2|\Delta^n Y_s|^2ds\quad\text{and}\nonumber\\
&\E\int_0^Te^{\int_0^s  \eta (\tau)d\tau}\chi_{\{\Delta^n Y_s> 0\}}\| U_s-\E_n U_s\|^2ds \leq \E\int_0^Te^{\int_0^s  \eta (\tau)d\tau}\chi_{\{\Delta^n Y_s> 0\}}\beta(s)^2|\Delta^n Y_s|^2ds,
\end{align} 
since the left hand sides tend to zero, while the right hand sides converge to $\delta_y$.
The same estimates hold for $Z'$ and $U'$ as well.

Hence, applying \eqref{anti-jensen} and \eqref{anti-jensen2} to \eqref{estimate-conditionaloutside} yields

\begin{align}\label{estimate-conditionalinside}
\E e^{\int_0^t \eta(s)ds}(\Delta^n Y_t)^2_+
&\leq \ \E\int_t^T e^{\int_0^s  \eta (\tau)d\tau}\chi_{\{\Delta^n Y_s> 0\}}\biggl[4\alpha(s){\rho}((\Delta^n Y_s)_+^2)\nonumber\\
&\quad+18\beta(s)^2(\Delta^n Y_s)^2_+-\eta(s)|\Delta^n Y_s|^2+\int_{_{B_n^c}}(\Delta^n Y_s)^2_+\nu(dx)\biggr]ds.
\end{align}

\textbf{Step 4:}

Bounding $\int_{B_n^c(s)}\!(\Delta^n Y_s)^2_+\nu(dx)$ by $\nu(\{1/n\leq |x|\})(\Delta^n Y_s)^2_+$ in \eqref{estimate-conditionalinside}, leads us to
\begin{align*}
\E e^{\int_0^t \eta(s)ds}(\Delta^n Y_t)^2_+
&\leq \ \E\int_t^T e^{\int_0^s  \eta (\tau)d\tau}\chi_{\{\Delta^n Y_s> 0\}}\biggl[4\alpha(s){\rho}((\Delta^n Y_s)_+^2)\nonumber\\
&\quad+(\nu(\{1/n\leq |x|\})+18\beta(s)^2)(\Delta^n Y_s)^2_+-\eta(s)|\Delta^n Y_s|^2\biggr]ds.
\end{align*} 
It remains, recalling that $\eta = 18 \beta^2$,
\begin{align*}
&\E e^{\int_0^t \eta (s)ds}(\Delta^n Y_t)^2_+\\
&\ \ \leq\E\int_t^T e^{\int_0^s \eta (\tau)d\tau}\biggl[4\alpha(s){\rho}((\Delta^n Y_s)_+^2)+\nu(\{1/n\leq |x|\})(\Delta^n Y_s)^2_+\biggr]ds\nonumber\\
&\ \ \leq \E\int_t^T e^{\int_0^s \eta (\tau)d\tau}\biggl[4(\alpha(s)\vee 1){\left(\rho+\nu(\{1/n\leq |x|\})\mathrm{id}\right)}((\Delta^n Y_s)_+^2)ds.
\end{align*}
The term $e^{\int_0^T \eta(\tau)d\tau}$ is $\mathbb{P}$-a.s. bounded by a constant $C>0$. Thus, by the concavity of $\rho_n:=\rho+\nu(\{1/n\leq |x|\})\mathrm{id}$, which satisfies the same assumptions as $\rho$, we arrive at
\begin{align*}
&\E(\Delta^n Y_t)^2_+\leq\E e^{\int_0^t \eta (s)ds}(\Delta^n Y_t)^2_+\leq\int_t^T 4C(\alpha(s)\vee 1){\rho_n}(\E(\Delta^n Y_s)_+^2)ds.\nonumber
\end{align*}
Then, the Bihari-LaSalle inequality (\autoref{bihari-prop}) shows that $\E(\Delta^n Y_t)^2_+=0$ for all $t\in{[0,T]}$.

\bigskip
\textbf{Step 5:}

\changed{Steps 1-4 granted that $\E_n Y_t \leq \E_n Y_t'$ for $n$ greater than a certain value, for all $t \in [0,T]$. By martingale convergence, $\E_n Y_t$ converges almost surely to the solution $Y_t$ of $(\xi,f)$ at time $t$ and  $\E_n Y_t'$ converges to the solution $Y_t'$ of $(\xi',f')$. Hence, in the limit we have $Y_t \leq Y_t',$ and the theorem is proven.}
\end{proof}

In the following \autoref{comparisonle2} we state a version of the above theorem for the case $1<p<2$. The difference to \autoref{comparison} is that here we cannot compare the generators on the solution only. If one wants to keep the comparison of the generators on the solution but accepts a slightly stronger condition than \changed{\eqref{gamma}}, given as ($\mathrm{H}_{\mathrm{comp}}$)in \cite{Kruse},
 \begin{itemize}[label={(A$\gamma$')}]
	\item 
	\quad $f(s,y,z,u)- f(s,y,z,u')\leq \int_{\R_0} (u(x) - u'(x)) \gamma_t(x) \nu(dx), \quad\mathbb{P}\otimes\lambda$-a.e.\newline
	\quad for a predictable process $\gamma = \gamma^{y,z,u,u'}$, such that $-1 \le \gamma_t(x)$ and $|\gamma_t(u)| \le \vartheta(u)$, where $\vartheta \in L^2(\nu)$,
\end{itemize} 
 then the proof of \cite[Proposition 4]{Kruse} can also be conducted for generators satisfying the conditions \ref{A3ge2} or \ref{A3le2}.

\begin{theorem}[Comparison, $p>1$]\label{comparisonle2}
	Let $p, p' > 1$ and $(Y,Z,U)$ be the $\lp$-solution to $(\xi,f)$ and $(Y',Z',U')$ be the $L^{p'}$-solution to $(\xi',f')$. Furthermore let $f$ and $f'$ satisfy \ref{A1} and \ref{A3ge2} or \ref{A3le2} for the according $p,p'$. If the following assumptions hold
	\begin{enumerate}[label=({\roman*})]
		\item
		$\xi\leq \xi'$, $\mathbb{P}$\text{-a.s.},
		\item
		$f(s,y,z,u)\leq f'(s,y,z,u)$, for all $(y,z,u)\in\mathbb{R}\times\mathbb{R}\times L^2(\nu),$ for $\mathbb{P}\otimes\lambda$\text{-a.a.} $(\omega,s)\in\Omega\times{[0,T]}$   
		and
		\item[(A$\gamma$)] for all $u,u'\in L^2(\nu)$ with $u'\geq u$
		\begin{align}
		\label{gammap1} 
		f(s,y,z,u)- f(s,y,z,u')\leq \int_{\mathbb{R}_0}(u'(x)-u(x))\nu(dx), \quad\mathbb{P}\otimes\lambda\text{-a.e},
		\end{align}
	\end{enumerate} 	
	then for all $t \in [0,T]$, we have $\mathbb{P}$\text{-a.s.}, 
	$$
	Y_t\leq Y'_t.
	$$
	
	The same assertion follows from an equivalent formulation for $f'$, requiring that \eqref{gammap1} \changed{holds} for $f'$.
\end{theorem}

\begin{proof}

We approximate the generators and terminal conditions by
$\xi_n:=\frac{n}{|\xi|\vee n}\xi, \xi'_n:=\frac{n}{|\xi'|\vee n}\xi'$ and 
\begin{align*}
&f_n(t,y,z,u):=\frac{n}{|f(t,y,z,u)|\vee n}f(t,y,z,u),\\ &f'_n(t,y,z,u):=\frac{n}{|f'(t,y,z,u)|\vee n}f'(t,y,z,u).
\end{align*}
This procedure preserves order relations. Furthermore, the generators $f_n, f'_n$ satisfy \ref{A1}, \ref{A3ge2} or \ref{A3le2} with respect to \changed{their coefficients}. Also \eqref{gammap1} remains satisfied for $f_n$. 

\changed{
Thus, the solutions $Y_n$ and $Y'_n$ of all these equations satisfy for all $t\in {[0,T]}$ $Y_{n,t}\leq Y'_{n,t}, \mathbb{P}\text{-a.s.}$ since (by the boundedness of $\xi_n,\xi'_n,f_n, f'_n$ and since \autoref{remA3}, \ref{remA3le2} implies \ref{A3ge2} for $p=2$ for the $f_n, f'_n$) they are also $L^2$-solutions. In the following steps, we will show convergence of those solutions to $Y$ in $\mathcal{S}^{p\wedge 2}$ and $Y'$ in $\mathcal{S}^{p'\wedge 2}$. Since for $p>2$ (or $p'>2$) the according solution is also an $L^2$-solution, convergence in $\mathcal{S}^2$ will suffice. Therefore, we will concentrate in the sequel on the case $p\leq 2$ for $Y$ (the case for $p'$ and $Y'$ works the same way).\bigskip

\textbf{Step 1:}

Consider the difference of the BSDEs for $Y_n$ and $Y$ (the convergence $Y'_n\to Y'$ can be shown in exactly the same way), with solutions $(Y_n,Z_n,U_n)$ and $(Y,Z,U)$,
\begin{align*}
Y_{n,t}-Y_t=&\xi_n-\xi+\int_t^T\left(f_n(s,Y_{n,s},Z_{n,s},U_{n,s})-f(s,Y_s,Z_s,U_s)\right)ds-\int_t^T(Z_{n,s}-Z_s)dW_s\\
&-\int_{]t,T]\times\R_0}\left(U_{n,s}(x)-U_s(x)\right)\tilde{N}(ds,dx),
\end{align*}
which can be written as a BSDE for $(Y_n-Y, Z_n-Z, U_n-U)$ using the generator
$$(s,y,z,u)\mapsto f_n(s,y+Y_s,z+Z_s,u+U_s)-f(s,Y_s,Z_s,U_s).$$
Then, \autoref{supproppleq2}, whose assumptions \ref{A1}, \ref{a3le2} are met for $\xi_n-\xi$ and this generator, yields 
\begin{align*}
&\|Y_n-Y\|_{\mathcal{S}^p}^p+\|Z_n-Z\|_{L^p(W)}^p+\|U_n-U\|_{L^p(\tilde{N})}\\
&\leq h\left(\E|\xi_n-\xi|^p+\E\left[\int_0^T\left|f_n (s,Y_{s},Z_{s},U_{s})-f(s,Y_s ,Z_s, U_s)\right|ds\right]^p\right).
\end{align*}
Clearly $\xi_n$ converges to $\xi$ in $L^p$. To obtain the desired convergence of $Y_n$ to $Y$ we need that $\E\left[\int_0^T\left|f_n (s,Y_{s},Z_{s},U_{s})-f(s,Y_s ,Z_s, U_s)\right|ds\right]^p\to 0$.

This expression can be written as 
 \begin{align*}
&\E\left[\int_0^T\left|f_n (s,Y_{s},Z_{s},U_{s})-f(s,Y_s ,Z_s, U_s)\right|ds\right]^p\\
&=\E\left[\int_0^T\left|f_n (s,Y_{s},Z_{s},U_{s})-f(s,Y_s ,Z_s, U_s)\right|\chi_{\{|f(s,Y_s ,Z_s, U_s)|>n\}}ds\right]^p,
 \end{align*}
which we can bound by
\begin{align*}
\E\left[\int_0^T2\left|f(s,Y_s ,Z_s, U_s)\right|\chi_{\{|f(s,Y_s ,Z_s, U_s)|>n\}}ds\right]^p.
\end{align*}
To show that the above sequence tends to $0$, it remains to show that
\begin{align*}
\E\left[\int_0^T\left|f(s,Y_s ,Z_s, U_s)\right|ds\right]^p<\infty,
\end{align*}
which we will do by a method appearing in the Brownian setting for  scalar BSDEs in \cite[Remark 4.3]{briand2003lp}. This will be the content of the next step.}
\bigskip

\changed{\textbf{Step 2:}

The limit version of It\^o's formula (see \cite[Corollary 1]{Kruse} and the estimate \cite[Lemma 9]{Kruse}), applied to the modulus function $|\cdot|$ and the BSDE $(\xi,f)$, together with
 immediate estimates yields 
\begin{align*}
|Y_0|\leq &|\xi|+\int_0^T\mathrm{sign}(Y_s)\left(f(s,Y_s,Z_s)-f_0(s)\right)ds+\int_0^T\mathrm{sign}(Y_s)f_0(s)ds\\
&-p\int_0^T\mathrm{sign}(Y_s)Z_sdW_s-p\int_{]0,T]\times\R_0}\mathrm{sign}(Y_{s-})U_s(x)\tilde{N}(ds,dx).
\end{align*}
Hence, for the positive and negative part of $\mathrm{sign}(Y_s)\left(f(s,Y_s,Z_s)-f_0(s)\right)$ we get
\begin{align*}
&\int_0^T\left[\mathrm{sign}(Y_s)\left(f(s,Y_s,Z_s)-f_0(s)\right)\right]_-ds\\
&\leq |\xi|+\int_0^T\left[\mathrm{sign}(Y_s)\left(f(s,Y_s,Z_s)-f_0(s)\right)\right]_+ds+\int_0^T\left|f_0(s)\right|ds\\
&\quad-p\int_0^T\mathrm{sign}(Y_s)Z_sdW_s-p\int_{]0,T]\times\R_0}\mathrm{sign}(Y_{s-})U_s(x)\tilde{N}(ds,dx).
\end{align*}
Adding the integral $\int_0^T\left[\mathrm{sign}(Y_s)\left(f(s,Y_s,Z_s)-f_0(s)\right)\right]_+ds$ on both sides, we get 
 \begin{align*}
&\int_0^T\left|f(s,Y_s,Z_s)-f_0(s)\right|ds\\
&\leq |\xi|+2\int_0^T\left[\mathrm{sign}(Y_s)\left(f(s,Y_s,Z_s)-f_0(s)\right)\right]_+ds+\int_0^T\left|f_0(s)\right|ds\\
&\quad-p\int_0^T\mathrm{sign}(Y_s)Z_sdW_s-p\int_{]0,T]\times\R_0}\mathrm{sign}(Y_{s-})U_s(x)\tilde{N}(ds,dx).
\end{align*}
By \ref{A3le2}, we come to
\begin{align*}
&\int_0^T\left|f(s,Y_s,Z_s)-f_0(s)\right|ds\\
&\leq |\xi|+2\int_0^T \chi_{\{ Y_s \neq 0\}}\left(\alpha(s)\rho\left(|Y_s|^p\right)|Y_s|^{1-p}+\mu(s)|Y_s|+\beta_1(s)|Z
_s|+\beta_2(s)\|U_s\|\right)ds+\int_0^T\left|f_0(s)\right|ds\\
&\quad-p\int_0^T\mathrm{sign}(Y_s)Z_sdW_s-p\int_{]0,T]\times\R_0}\mathrm{sign}(Y_{s-})U_s(x)\tilde{N}(ds,dx).
\end{align*}
Since $\lim_{x\to 0}\tfrac{\rho(x^p)}{x^{p-1}}=0$, there is a value $r>0$ such that $\tfrac{\rho(x^p)}{x^{p-1}}\leq 1$, for $x \le r$. Additionally, from the concavity of $\rho$ it follows that for $x\geq r$ we have $\rho(x)\leq\tfrac{\rho(r)}{r}x$. This means that for $x\geq 0$,
\begin{align*}
\rho(x^p)x^{1-p}\leq 1\chi_{\{x\leq r\}}+\frac{\rho(r)}{r}x\chi_{\{x\geq r\}}\leq 1+\frac{\rho(r)}{r}x.
\end{align*}}
\changed{
With this inequality for the function $\rho$ we arrive at
\begin{align*}
&\int_0^T\left|f(s,Y_s,Z_s)-f_0(s)\right|ds\\
&\leq |\xi|+2\!\int_0^T\!\!\alpha(s)ds+2\!\int_0^T\!\!\left(\!\left(\alpha(s)\frac{\rho(r)}{r}+\mu(s)\right)|Y_s|+\beta_1(s)|Z
_s|+\beta_2(s)\|U_s\|\right)ds+\!\int_0^T\!\left|f_0(s)\right|ds\\
&\quad+p\left|\int_0^T\mathrm{sign}(Y_s)Z_sdW_s\right|+p\left|\int_{]0,T]\times\R_0}\mathrm{sign}(Y_{s-})U_s(x)\tilde{N}(ds,dx)\right|.
\end{align*}
We take the $p$-th power and expectations, then estimate with a constant $c>0$,
\begin{align}\label{eq:intgeneratorpfinite}
&\E\left[\int_0^T\left|f(s,Y_s,Z_s)-f_0(s)\right|ds\right]^p\notag\\
&\leq c\E\Biggl(1+|\xi|^p+\left[\int_0^T\left(\alpha(s)\frac{\rho(r)}{r}+\mu(s)\right)|Y_s|ds\right]^p+\left[\int_0^T\beta_1(s)|Z
_s|ds\right]^p+\left[\int_0^T\beta_2(s)\|U_s\|ds\right]^p\notag\\
&\quad+\left[\int_0^T\left|f_0(s)\right|ds\right]^p+\left|\int_0^T\mathrm{sign}(Y_s)Z_sdW_s\right|^p+\left|\int_{]0,T]\times\R_0}\mathrm{sign}(Y_{s-})U_s(x)\tilde{N}(ds,dx)\right|^p\Biggr).
\end{align}
By the Cauchy-Schwarz inequality, 
\begin{align*}
&\E\left[\int_0^T\beta_1(s)|Z
_s|ds\right]^p+\E\left[\int_0^T\beta_2(s)\|U_s\|ds\right]^p\\
&\leq\E\left[\int_0^T \beta_1(s)^2ds\right]^\frac{p}{2}\left[\int_0^T |Z_s|^2ds\right]^\frac{p}{2}+\E\left[\int_0^T \beta_2(s)^2ds\right]^\frac{p}{2}\left[\int_0^T \|U_s\|^2ds\right]^\frac{p}{2}\\
&\leq c_1\E\left(\left[\int_0^T |Z_s|^2ds\right]^\frac{p}{2}+\left[\int_0^T \|U_s\|^2ds\right]^\frac{p}{2}\right)<\infty,
\end{align*}
for some constant $c_1>0$, and further,
\begin{align*}
&\E\left[\int_0^T\left(\alpha(s)\frac{\rho(r)}{r}+\mu(s)\right)|Y_s|ds\right]^p\leq \E\sup_{s\in{[0,T]}}|Y_s|^p\left[\int_0^T\left(\alpha(s)\frac{\rho(r)}{r}+\mu(s)\right)ds\right]^p\\
&\leq c_2\E\sup_{s\in{[0,T]}}|Y_s|^p<\infty,
\end{align*}
for some $c_2>0$.
The finiteness of $\E\left|\int_0^T\mathrm{sign}(Y_s)Z_sdW_s\right|^p$ follows from the Burkholder-Davis-Gundy inequality and the one of $\E\left|\int_{]0,T]\times\R_0}\mathrm{sign}(Y_{s-})U_s(x)\tilde{N}(ds,dx)\right|^p$ follows from \cite[Theorem 3.2]{MarinRoeck}.
As all terms of the right hand side in \eqref{eq:intgeneratorpfinite} are now finite, we get that
\begin{align*}
\E\left[\int_0^T\left|f(s,Y_s,Z_s,U_s)-f_0(s)\right|ds\right]^p<\infty .
\end{align*}
The finiteness of $\E\left[\int_0^T|f_0(s)|ds\right]^p$ implies that $\E\left[\int_0^T|f(s,Y_s,Z_s,U_s)|ds\right]^p<\infty$, as desired. Therefore, with step 1 we conclude that the convergence 
\begin{align*}
\E\left[\int_0^T|f_n(s,Y_s,Z_s,U_s)-f(s,Y_s,Z_s,U_s)|ds\right]^p\to 0
\end{align*}
takes place which proves that $Y_n\to Y$ in $\mathcal{S}^{p\wedge 2}$.\bigskip

 Then, for all $t$,
$$Y_{n,t}\to Y_{t},\quad Y'_{n,t}\to Y'_{t},\quad \mathbb{P}\text{-a.s.,}$$
 and $Y_t\leq Y'_t$, $\mathbb{P}$-a.s. follows.} 
\end{proof}

\begin{acknowledgements}
The authors thank Christel Geiss, University of Jyv\"askyl\"a, and Gunther Leobacher, University of Graz, for fruitful discussions and valuable suggestions.\smallskip

Stefan Kremsner and Alexander Steinicke are supported by the Austrian Science Fund (FWF): Project F5508-N26, which is part of the Special Research Program ``Quasi-Monte Carlo Methods: Theory and Applications''.
\end{acknowledgements}

\section*{Appendix}\label{Appendix}

\subsection{Inequalities}\label{Appendix:Inequalities}

\begin{theorem}[Young's inequality] \label{young_inequality}
	For $R > 0$ and $a,b \in \R$, we can bound the product
	\begin{align*}
	ab = aR^{-\frac{1}{p}} R^{\frac{1}{p}} b \leq \frac{a^p}{pR}+\frac{b^qR^{\frac{q}{p}}}{q}\,
	\end{align*}
	with $p,q>1$ satisfying $\frac{1}{p}+\frac{1}{q}=1$.
\end{theorem}

{\bf The Bihari-LaSalle inequality.} 
For the  Bihari-LaSalle inequality   we refer to \cite[pp. 45-46]{maobook}. Here we formulate a  backward version of it which  has been applied  in \cite{YinMao}.
The proof is analogous to that in \cite{maobook}.

\begin{theorem} \label{bihari-prop} Let $c>0.$ Assume that $\rho: [0,\infty[ \to [0,\infty[$ is a continuous and nondecreasing function such that $\rho(x) >0$ for all 
	$x>0.$ Let $K$ be a nonnegative, integrable Borel function on $[0,T],$ and $y$ a nonnegative, bounded Borel function on $[0,T],$ such that
	\begin{align*}
	y(t) &\le c + \int_t^T K(s) \rho(y(s)) ds.
	\end{align*}   
	Then for $z(t) := c + \int_t^T K(s) \rho(y(s)) ds$
	it holds that
	$$ y(t) \le G^{-1} \left (G(c) +   \int_t^T K(s)  ds \right ) $$
	
	for all $t \in [0,T]$ such that $G(c) +   \int_t^T K(s)  ds \in {\rm dom}(G^{-1}).$
	Here 
	\begin{align*}
	G(x) := \int_1^x \frac{dr}{\rho(r)},
	\end{align*}
	and $G^{-1}$ is the inverse function of $G.$
	
	Furthermore, if for an $\epsilon > 0$, $\int_0^\epsilon \frac{dr}{\rho(r)} = \infty$, then $$y(t)=0$$ for all $t\in [0,T]$.
\end{theorem}

\begin{remark}
	As a special case we have Gronwall's inequality: If $\rho(r)=r$ for $r\in [0,\infty[,$ we get
	\equa
	y(t)\le c e^{\int_t^T K(s)  ds}.
	\tion
\end{remark}

\subsection{Construction of predictable and progressively measurable versions}\label{Appendix:sec2}

We will use the next Lemma for the construction of progressively measurable versions of conditional expectations necessary for the proof of \autoref{comparison}. To prepare it, we need the following definitions, appearing also in \cite[Definition 3.2]{geiss2018monotonic} :\medskip

Assume that $\mathbbm{F}:= \left(\mathcal{F}^{[s]}\right)_{s\in[0,\infty[}$  is built using the definitions from \autoref{sec:finlev}, where $[\cdot]$ denotes the floor function.  
Recall that $\mathcal{P}$ denotes the predictable $\sigma$-algebra according to the filtration $(\mathcal{F}_t)_{t\in [0,T]}$.
Let $(V,\mathcal{V},\mu)$ be a $\sigma$-finite measure space and $K\colon\Omega\times{[0,T]}\times V\to\mathbb{R}$ be a bounded, $\mathcal{P}\otimes\mathcal{V}$-measurable process. 
Define $\phantom{I}^{\!\!\!o,\mathbbm{F}\!\!} K$ to be the optional projection of the process 
\begin{align*}
{[0,\infty[}\times\Omega\times{[0,T]}\times V \to&\ \ \R,\\
(s,\omega,t,x)\mapsto& \ K(\omega,t,x)
\end{align*} 
in the variables $(s,\omega)$ with respect to 
$\mathbbm{F},$ and with  parameters $(t,x).$
For each $n\ge 0$, assume that the filtration $\mathbbm{F}^n :=\left(\mathcal{F}_t^n\right)_{t\in{[0,T]}}$ is given by $\mathcal{F}_t^n:=\mathcal{F}_t\cap\mathcal{F}^n.$
Let $K_n$ be the predictable projection of 
$$(\omega,t,x)\mapsto \phantom{I}^{\!\!\!o,\mathbbm{F}\!\!} K(n,\omega,t,x)$$
 with respect to  $\mathbbm{F}^n$ with parameter $x$. 

\bigskip
The reason for using the family of filtrations  $\left(\mathcal{F}^{[s]}\right)_{s\in[0,\infty[}$  instead of the sequence $(\mathcal{F}^n)_{n=0}^\infty$  from \autoref{sec:finlev} is  
that one can apply  known measurability  results w.r.t.~right continuous filtrations instead of proving measurability here directly. Indeed,
the optional projection  $\phantom{I}^{\!\!\!o,\mathbbm{F}\!\!}K$ defined above  is jointly measurable in $(s,\omega,t,x).$  For this we refer to  \cite{meyer}, 
where predictable and optional projections of random processes depending on parameters were considered, and their uniqueness up to indistinguishability
was shown. 

It follows that for all $(t,x) \in [0,T] \times V$, $$\phantom{I}^{\!\!\!o,\mathbbm{F}\!\!}K(n,t,x)=\E_n K(t,x), \quad \mathbb{P}\text{-a.s.} $$  

Then, since $K$ is $\left(\mathcal{F}_t\right)_{t\in{[0,T]}}$-predictable, for all $n\geq 0$, $t\in {[0,T]}$ and all $x \in V$, it holds that
\begin{equation}\label{optionalcond}
K_n(t,x)=\E_n K(t,x), \quad\mathbb{P}\text{-a.s.}
\end{equation}
Hence $K_n(t,x)$ is a jointly measurable version of $\E_n K(t,x)$ which  is $\left(\mathcal{F}_t^n\right)_{t\in{[0,T]}}$-predictable, so in particular it is predictable. \\

To extend the above construction to unbounded processes $K$, we have the following result:

\changed{
\begin{lemma}\label{measurable-version}
	Let $(V,\mathcal{V},\mu)$ be a $\sigma$-finite measure space and $K\colon\Omega\times{[0,T]}\times V\to\mathbb{R}$ be a $\mathcal{P}\otimes\mathcal{V}$-measurable process such that for $\lambda\otimes\mu$-almost all $(t,x)\in {[0,T]}\times V$,
	\begin{align*}
	\E |K(t,x)|<\infty.
	\end{align*}
	
There is a $\mathcal{P}\otimes\mathcal{V}$-measurable process $\left(K_n(t,x)\right)_{(t,x)\in [0,T]\times V}$, and $$K_n(\omega,t,x)=\E_n K(\omega,t,x), \mathbb{P}\text{-a.s. for }\lambda\otimes\mu\text{-a.a. }(t,x)\in [0,T]\times V.$$
\end{lemma}}

\begin{proof}
For $m \ge 1$ let $K^m:=\frac{m}{|K|\vee m}K$ be a bounded approximation for $K$. For each $m$, we can construct the process $K^m_n$, as described above the lemma. We intend to let $m\to\infty$. By the definition of the optional and predictable projections, we have that for all $(t,x) \in [0,T] \times V$, \begin{align*}
K^m_n(t,x)=\E_{t-}\E_n K^m(t,x),\quad\mathbb{P}\text{-a.s.,}
\end{align*} 
where we emphasize that $\E_{t-}$ is the conditional expectation with respect to the $\sigma$-algebra $\mathcal{F}_{t-}=\sigma\left(\bigcup_{0\leq s<t}\mathcal{F}_s\right)$ and $\E_n$ denotes the expectation with respect to the $\sigma$-algebra $\mathcal{F}^n$.
Using the assumption $\E |K(t,x)|<\infty$ for $\lambda\otimes\mu$-a.a. $(t,x)$, we find by dominated convergence that for a set $A\in \mathcal{B}([0,T])\otimes\mathcal{V}$ with $(\lambda\otimes\mu)(A)<\infty$,
\begin{align*}
&\int_A \E \left(|K^{m_1}_n(t,x)-K^{m_2}_n(t,x)|\wedge 1\right)(\lambda\otimes\mu)(dt, dx)\\
&=\int_A \E \left(|\E_{t-}\E_n K^{m_1}(t,x)-\E_{t-}\E_nK^{m_2}(t,x)|\wedge 1\right) (\lambda\otimes\mu)(dt, dx)\\
&\leq \int_A \E \left(|K^{m_1}(t,x)-K^{m_2}(t,x)|\wedge 1\right) (\lambda\otimes\mu)(dt, dx) \to 0\quad\text{as }m_2>m_1\to\infty.
\end{align*}

Since $(V,\mathcal{V},\mu)$ is $\sigma$-finite, it follows that the sequence $(K^m_n)_{m\geq 1}$ converges also locally in the measure $\mathbb{P}\otimes\lambda\otimes\mu$, from which we then infer that we can extract a subsequence $(K^{m_k}_n)_{k\geq 1}$ that converges $\mathbb{P}\otimes\lambda\otimes\mu$-a.e. 
Without loss of generality we assume that $(K^m_n)_{m\geq 1}$ itself converges $\mathbb{P}\otimes\lambda\otimes\mu$-a.e.
This means that one can find an $\mathcal{F}\otimes\mathcal{B}([0,T])\otimes\mathcal{V}$-measurable process $\tilde K_n$ and a set $M\in \mathcal{F}\otimes\mathcal{B}([0,T])\otimes\mathcal{V}$ with $\mathbb{P}\otimes\lambda\otimes\mu\left((\Omega\times{[0,T]}\times V)\setminus M\right)=0$ such that for all $(\omega,t,x)\in M$, $\lim_{m\to\infty}K^m_n(\omega,t,x)=\tilde K_n(\omega,t,x)$. On the $\sigma$-algebra 
\begin{align*}
(\mathcal{P}\otimes\mathcal{V})_M:=\left\{A\cap M : A\in \mathcal{P}\otimes\mathcal{V}\right\},
\end{align*}
we have that the process $\tilde K_n$ restricted to the set $M$ is the limit of the processes $(K^m_n)$ restricted to $M$ and is therefore $(\mathcal{P}\otimes\mathcal{V})_M$-measurable. In the sense of \cite[Theorem 1]{shortt} we can then extend the restriction of the process $\tilde K_n$ to a $\mathcal{P}\otimes\mathcal{V}$-measurable process $K_n$. 

To show that the process $K_n$ satisfies condition \eqref{optionalcond} for $\lambda\otimes\mu$-a.a. $(t,x)\in [0,T]\times V$, one can use the dominated convergence theorem for conditional expectations, 
\begin{align*}
K_n(t,x)=\lim_{m\to\infty} K^m_n(t,x)=\lim_{m\to\infty}\E_{t-}\E_n K^m(t,x)=\E_{t-}\E_n\lim_{m\to\infty} K^m(t,x)=\E_{t-}\E_n K(t,x),
\end{align*}
where the latter equals $\E_nK(t,x), \lambda\otimes\mu$-a.e. since $K$ is a predictable process.
\end{proof}  

\begin{remark}\label{measurable-version-remark}
\begin{enumerate}[label=({\roman*})]
\item A similar proof as above can be done if one considers optional processes instead of predictable ones.
	If the process $K$ does not depend on an additional parameter set $V$ and is such that $\E|K(t)|<\infty$, for all $t$, then again very similar arguments as in the above Lemma can be used to find that there is a progressively measurable process $K_n$ with
	$\E_nK(t)=K_n(t)$ for all $t\in [0,T]$ (this case is the one needed for a progressively measurable version of $(\E_n Y_t)_{t\in [0,T]}$ in the proof of Theorem \ref{comparison}).
	\item A different way of proving \autoref{measurable-version} is to represent the process $K$ by a functional $F^K\colon D[0,T]\times[0,T]\times V$, where $D[0,T]$ denotes the space of c\`adl\`ag functions on $[0,T]$ endowed with the $\sigma$-algebra generated by the projection maps $p_t\colon D[0,T]\to\mathbb{R}, \mathrm{x}\mapsto\mathrm{x}_t$. The measure we consider on this sigma field is the pushforward measure $\mathbb{P}_X$ of the L\'evy process $X$, given by the trajectory mapping $X\colon\Omega\to D[0,T], \omega\mapsto X(\omega)$. We have that $K(t,x)=F^K(X,t,x)$ up to indistinguishability (for details on this representation see \cite{SteinickeI}). Denoting the L\'evy process with cut-off L\'evy measure again by $X^n$ (see \autoref{sec:finlev}), one can then show that the process $G$ given by 
	\begin{align*}
		G(t,x)=\begin{cases}\E F^K(h+X-X^n,t,x)\Big|_{h=X^n}, &\text{whenever the expectation exists and is finite,}\\
	0, &\text{else},\end{cases}
	\end{align*}
	possesses a predictable version that satisfies the assertion of \autoref{measurable-version}.
\end{enumerate}
	
\end{remark}

%
%



\end{document}